\documentclass[11pt]{article}

\usepackage[final]{showkeys}
\usepackage[obeyspaces,hyphens,spaces]{url}

\renewcommand*\showkeyslabelformat[1]{%
\fbox{\parbox[t]{1.4 cm}{\raggedright\normalfont\small\url{#1}}}}

\usepackage{amsthm, amsmath, amsfonts, amssymb, mathtools, wasysym,  mathrsfs, empheq, etoolbox} 
\usepackage{mathpazo}
\usepackage[mathscr]{eucal} 

\usepackage[top= 2 cm, bottom = 2 cm, left = 2 cm, right= 2 cm]{geometry}
\usepackage[title]{appendix}

\usepackage[table, dvipsnames]{xcolor} %

\usepackage[labelsep=period]{caption}  
\usepackage{subcaption}

\usepackage{graphicx, soul, float}

\usepackage{todonotes}

\definecolor{labelkey}{rgb}{.1,.1,.8}
\definecolor{refkey}{rgb}{0,0.6,0.0}

\definecolor{dgreen}{rgb}{0.00,0.49,0.00}
\definecolor{dblue}{rgb}{0,0.08,0.75}
\colorlet{myblue}{dblue}
\colorlet{mygreen}{dgreen}

\definecolor{myfirstblue}{rgb}{.8, .8, 1}

\newcommand*\mybluebox[1]{%
    \colorbox{RoyalBlue!20}{\hspace{1em}#1\hspace{1em}}}

\usepackage{hyperref}
\hypersetup{ 
	linktocpage=true,
	colorlinks= true, 
	urlcolor = Red, 
	citecolor = mygreen,
	linkcolor = myblue
}

\allowdisplaybreaks

\usepackage[capitalize, nameinlink]{cleveref} 

\crefdefaultlabelformat{#2\textup{#1}#3}
\crefrangeformat{enumi}{#3\textup{#1}#4\textup{\textendash}#5\textup{#2}#6}
\crefname{equation}{}{}

\crefname{item}{}{items}
\crefname{figure}{Figure}{Figures}
\crefname{theorem}{\protect\theoremname}{\protect\theoremname}
\crefname{lemma}{\protect\lemmaname}{\protect\lemmaname}
\crefname{proposition}{\protect\propositionname}{\protect\propositionname}
\crefname{corollary}{\protect\corollaryname}{\protect\corollaryname}
\crefname{fact}{\protect\factname}{\protect\factname}
\crefname{example}{\protect\examplename}{\protect\examplename}
\crefname{remark}{\protect\remarkname}{\protect\remarkname}
\crefname{enumi}{}{}
\crefname{appsec}{Appendix}{Appendices}

\makeatletter

\let\orgdescriptionlabel\descriptionlabel
\renewcommand*{\descriptionlabel}[1]{%
	\let\orglabel\label
	\let\label\@gobble
	\phantomsection
	\edef\@currentlabel{#1}%
	\let\label\orglabel
	\orgdescriptionlabel{#1}%
}

\let\leq\leqslant
\let\geq\geqslant
\let\subset\subseteq

\renewcommand{\implies}{\Rightarrow}
\renewcommand{\iff}{\Leftrightarrow}
\renewcommand{\impliedby}{\Leftarrow}

\def\th@plain{%
	\thm@notefont{} 
	\itshape 
}
\def\th@definition{%
	\thm@notefont{}
	\normalfont 
}

\g@addto@macro\th@remark{\thm@headpunct{}}
\g@addto@macro\th@definition{\thm@headpunct{}}
\g@addto@macro\th@plain{\thm@headpunct{}}

\makeatother

\theoremstyle{plain}

\newtheorem{theorem}{\protect\theoremname}[section]
\newtheorem{corollary}[theorem]{\protect\corollaryname}
\newtheorem{lemma}[theorem]{\protect\lemmaname}
\newtheorem{proposition}[theorem]{\protect\propositionname}

\theoremstyle{definition}

\newtheorem{remark}[theorem]{\protect\remarkname}
\newtheorem{example}[theorem]{\protect\examplename}
\newtheorem{fact}[theorem]{\protect\factname}
\newtheorem{openprob}[theorem]{\protect\openprobname}

\providecommand{\theoremname}{Theorem}
\providecommand{\propositionname}{Proposition}
\providecommand{\corollaryname}{Corollary}
\providecommand{\factname}{Fact}
\providecommand{\lemmaname}{Lemma}

\providecommand{\remarkname}{Remark}
\providecommand{\examplename}{Example}
\providecommand{\openprobname}{Open Problem}

\newcommand{\ball}[2]{\ensuremath{  \operatorname{B} \left({#1};{#2}\right) } }

\newcommand{\pc}[1]{{#1}^{\ominus}} 

\DeclarePairedDelimiterX\menge[2]{ \{ }{ \} }{ {#1} ~ \delimsize \vert ~ \mathopen{} {#2} }  
\DeclarePairedDelimiterX\fa[2]{ ( }{ )_{#2} }{#1}  
\DeclarePairedDelimiterX\set[2]{ \{ }{ \}_{#2} }{#1}  

\DeclarePairedDelimiterX\rb[1]{ ( }{ ) }{#1}
\DeclarePairedDelimiterX\rbr[1]{ ( }{ ) }{#1}
\DeclarePairedDelimiterX\sqbo[1]{ ] }{ [ }{#1}
\DeclarePairedDelimiterX\sqbc[1]{ [ }{ ] }{#1}

\DeclarePairedDelimiterX\scal[2]{\langle}{\rangle}%
{ \ifblank{#1#2}{ \, \cdot \, \delimsize \vert \, \mathopen{}\cdot \, } {   {#1} \, \delimsize \vert \, \mathopen{}{#2}  } }

\DeclarePairedDelimiterX\norm[1]{\lVert}{\rVert}%
{ \ifblank{#1}{\, \cdot \,}{#1} }

\DeclarePairedDelimiterX\abs[1]{\lvert}{\rvert}{\ifblank{#1}{\, \cdot \,}{#1}}

\newcommand{\HH}{\ensuremath{\mathcal H}}

\newcommand{\RR}{\ensuremath{\mathbb R}}
\newcommand{\NN}{\ensuremath{\mathbb N}}

\newcommand{\pinf}{\ensuremath{ {+}\infty }}
\newcommand{\minf}{\ensuremath{ {-}\infty }}
\newcommand{\RX}{\ensuremath{ \left]\minf, \pinf\right] }}
\newcommand{\RRX}{\ensuremath{ \left[ \minf,\pinf \right]}}

\newcommand{\RP}{\ensuremath{\mathbb{R}_{+}}}
\newcommand{\RPP}{\ensuremath{\mathbb{R}_{++}}}
\newcommand{\RMM}{\ensuremath{\mathbb{R}_{--}}}
\newcommand{\RM}{\ensuremath{\mathbb{R}_{-}}}

\newcommand{\closu}[1]{\ensuremath{\overline{#1} }}

\newcommand{\lspan}{\ensuremath{\operatorname{span}}}

\newcommand{\cdom}{\ensuremath{\operatorname{\overline{dom}}}}
\newcommand{\cran}{\ensuremath{\operatorname{\overline{ran}}}}
\newcommand{\ran}{\ensuremath{\operatorname{ran}}}
\newcommand{\dom}{\ensuremath{\operatorname{dom}}}

\newcommand{\gra}{\ensuremath{\operatorname{gra}}}
\newcommand{\Fix}{\ensuremath{\operatorname{Fix}}}
\newcommand{\Id}{\ensuremath{ {\operatorname{Id} }}}
\newcommand{\prox}{\ensuremath{\operatorname{Prox}}}
\newcommand{\card}{\ensuremath{\operatorname{card}}}

\newcommand{\q}{\ensuremath{ {\operatorname{q} }}  }
\newcommand{\lscc}{\ensuremath{\varGamma_{0}}}
\newcommand{\infconv}{\ensuremath{\mbox{\footnotesize$\,\square\,$}}}
\newcommand{\einfconv}{\ensuremath{\mbox{\footnotesize$\,\boxdot\,$}}}
\newcommand{\env}{\ensuremath{\operatorname{env}}}

\newcommand*{\stardiff}{%
	\mathrel{\vcenter{\offinterlineskip
			\hbox{$\kern+2.0pt {}^\ast$}\vskip-1.20ex\hbox{$-$}}}}

\let\originalleft\left
\renewcommand{\left}{\mathopen{}\originalleft}

\newcommand{\email}[1]{\href{mailto:#1}{\nolinkurl{#1}}} 

\newcommand{\prH}{\ensuremath{\mathrm{Proj}\rbr{\HH} }}

\let\setminus\smallsetminus

\begin{document}

\title{  
\sffamily	On
      sums and convex
      combinations
      of projectors
	onto convex sets
 }

\author{
  	Heinz H.\ Bauschke\thanks{
  		Mathematics, University of British Columbia, Kelowna, B.C.\ V1V~1V7, Canada. 
  	Email: 
      \href{mailto: heinz.bauschke@ubc.ca}{\texttt{heinz.bauschke@ubc.ca}}.},~
      Minh N. Bui\thanks{
        Department of Mathematics,  
North Carolina State University,
Raleigh,
    NC~27695-8205,~USA and
        Mathematics,
            University of British Columbia,
            Kelowna, B.C. \ V1V~1V7,
            Canada.
Email: \href{mailto: mnbui@ncsu.edu}{\texttt{mnbui@ncsu.edu}}.},~
  	and Xianfu Wang\thanks{
  	Mathematics, University of British Columbia, Kelowna, B.C.\ V1V~1V7, Canada. 
  	Email: \href{mailto: shawn.wang@ubc.ca}{\texttt{shawn.wang@ubc.ca}}.}
}
\date{September~14, 2018}


\maketitle

\begin{abstract}
\noindent
The projector
onto the  
Minkowski
sum 
of closed convex sets
is generally not
equal to
the sum of individual projectors.
In this work,
we provide a complete 
answer
to
the question
of characterizing
the instances
where such 
an equality holds.
Our results
unify and extend
the case of linear subspaces
and 
Zarantonello's results
for projectors onto cones. 
A detailed
analysis in the case
of convex combinations
is also carried out.
We establish 
the partial sum property
for projectors
onto convex cones, 
and we also present
various examples
as well as
a detailed analysis
in the univariate
case.
\end{abstract}
{\small
\noindent
{\bfseries 2010 Mathematics Subject Classification:}
{Primary 
47H05, 
47H09, 
47L07; 
Secondary 
46A32, 
52A05,  
52A41, 
90C25. 
}

\noindent {\bfseries Keywords:}
convex set,
convex cone,
convex combination,
projection operator,
projector, 
sum of projectors,
partial sum property,
monotone operator,
proximity operator.
}


\section{Introduction}

Throughout this paper,
we assume that 
\begin{empheq}[box = \mybluebox]{equation}\label{H}
	\text{$\HH$ is a real Hilbert space} 
\end{empheq}
with inner product $\scal{}{}$ 
and induced norm $\norm{}$.
Now assume  
that\footnote{For basic 
	Convex Analysis,
    we refer the reader to  \cite{bauschke2017convex,rocky,Zalinescu-book2002,zarontello1971projections-partI}.}
\begin{empheq}[box = \mybluebox]{equation}
\text{$\fa{C_{i}}{i \in I}$ 
	is a finite
	family
	of nonempty closed convex subsets
	of 
	$\HH$} 
\end{empheq}
with corresponding projectors
\begin{empheq}[box = \mybluebox]{equation}
\fa{P_{C_{i}}}{i \in I}
\end{empheq}
and that 
      \begin{empheq}[box =\mybluebox]{equation}
        \text{$\fa{\alpha_{i}}{i \in I}$ are real numbers.}
      \end{empheq}
In this paper, we analyze carefully 
the question:
\emph{When is $\sum_{i \in I}\alpha_{i} P_{C_{i}}$ 
  a projector?}
This allows us
to provide a complete answer
to the question 
``\emph{When is
the sum of projectors also a projector?}''
(In view of \cref{p:sumCi}\cref{i:set-sum2}, 
an affirmative answer to this question
requires the sum $\sum_{i\in I}C_i$ to be closed.
This happens, for instance, when each set is bounded.)
It is known that, in the case of
linear subspaces,
$\sum_{ i \in I} P_{C_{i}}$
is a projector
onto a closed linear subspace
if and only if
$	\fa{C_{i}}{i \in I}$
is pairwise orthogonal; 
see \cite[Theorem~2, p.~46]{Halmos-spectral-1951}. 
This question is also of interest
in Quantum Mechanics \cite[p.~50]{KayeQuantum}. 
In 1971, Zarantonello \cite{zarontello1971projections-partI}
answered 
this question in the case
of convex cones, i.e.,
if $\fa{C_{i}}{i \in I}$ are cones,
then
$\sum_{ i \in I} P_{C_{i}}$
is a projector
if and only if 
$\fa{P_{C_{i}}}{i \in I}$
is pairwise orthogonal
in the sense that,
for every $\rb{i,j} \in I \times I$
with $i \neq j$,  we have
$\rb{\forall x \in \HH}\,\scal{P_{C_{i}}x}{ P_{C_{j}}x } = 0.$ 
However, the question remains open 
in the general convex case.  
Therefore,
one goal of this paper 
is to provide
necessary and sufficient 
conditions
for $\sum_{ i \in I}\alpha_{i} P_{C_{i}}$
to be a projector
without any further assumption
on the sets $\fa{C_{i}}{ i \in I}$.
As a consequence,
we answer entirely the question ``When is
the sum of projectors also a projector?''
Our results 
unify the two aforementioned results
and make a connection with
the 
recent work 
\cite{bartz2017resolvent} 
where it was proven 
that, if the sum of a family of proximity operators
is a proximity operator, then every partial sum
remains a proximity operator. 
Interestingly, we shall see that
this property is still valid
in the class of projectors onto convex cones;  
in other words,
if a finite sum of 
projectors onto convex cones is
a projector, then so are its partial sums. 
Nevertheless, 
this result fails 
outside the world of convex cones.  
Another goal
is to characterize
the 
instances where
a convex average of $\fa{P_{C_{i}}}{i \in I}$
is again a projector. In striking contrast 
to a result in 1963 by Moreau
\cite{Moreau-prox-1963}, which
states that a convex average of
proximity operators is always
a proximity operator, 
we shall see in
\cref{t:convex-comb}
that taking convex combinations
does not preserve the class
of projectors onto convex sets
 (see \cref{t:convex-comb} for
  the rigorous statement).
Our main
results are summarized as follows:
\begin{itemize}
	\item We provide a 
        new characterization
        of proximity operators
        in \cref{t:prox-charac} (for
          a list of 
        other characterizations,
      see \cite{Combettes-2018-mono}).
          In turn,
          we derive a
          new characterization
          of projectors (\cref{t:main}),
          which is
	a pillar of this paper 
	and 
	a variant
	of \cite[Theorem~4.1]{zarontello1971projections-partI}.
	Furthermore, we also
	partially answer 
	an open question by Zarantonello
	regarding \cite[Theorem~4.1]{zarontello1971projections-partI}.
    \item \cref{t:conic-comb}
      characterizes
	(without any additional assumptions
	on the underlying sets)
      when $\sum_{i\in I}\alpha_{i} P_{C_{i}}$
      is a projector; 
      \cref{t:sum-many}       
      concerns the sum $\sum_{i \in I}P_{C_{i}}$.
    \item By specifying our 
      analysis to the case
      of convex average in \cref{t:convex-comb},
      we explicitly determine 
      families of closed convex sets
      that are preserved under
      taking convex combinations.
	\item We present the 
	\emph{partial sum property} 
	(see \cite[Theorem~4.2]{bartz2017resolvent})
	for projectors onto convex cones
	in \cref{c:cone-many},
	whose proof is based on
	\cref{t:two-cones-pr}
	and \cite[Theorem~4.2]{bartz2017resolvent}.
	We also recover 
	\cite[Theorems~5.3 and 5.5]{zarontello1971projections-partI}.
\end{itemize}

The paper is organized as follows. 
In \cref{sect:aux},
we collect miscellaneous results
that will be used in the sequel. 
Our main results are presented 
in \cref{sect:main}: 
\cref{t:prox-charac}
provides a characterization
of proximity operators,
while projectors 
are dealt with in \cref{t:main},
which is a variant of
\cite[Theorem~4.1]{zarontello1971projections-partI}. 
This allows
us to recover the classical
characterization of orthogonal projectors; 
see, e.g., \cite[Theorem~4.29]{weidmann2012}.
In turn, 
we establish
a necessary and sufficient 
condition
for a linear combination
of projectors 
to be a projector in \cref{t:conic-comb}
and then particularize 
to sums of projectors in \cref{t:sum-many}.
We then specialize 
the analysis of
\cref{sect:main}
to convex combinations
of projectors in \cref{sect:convex-comb}.
In \cref{sect:partial-sum}, 
we show that,
in the case of sums of projectors,
\cref{t:sum-many} covers
the result
obtained by Zarantonello 
(\cite[Theorem~5.5]{zarontello1971projections-partI})
and the case
of linear subspaces.
Furthermore, we provide 
\cref{t:two-cones-pr} and \cref{c:cone-many}
to illustrate the 
connection 
between our work
and  \cite{bartz2017resolvent,zarontello1971projections-partI}.
The one-dimensional
case is the topic
of \cref{sect:R},
where  all
the 
pairs
$\rb{C,D}$
of nonempty closed
convex subsets 
of $\RR$
satisfying $P_{C} + P_{D} = P_{C+D}$
are explicitly determined. 
Finally, we turn to 
a generalization 
of
the classical result 
\cite[Theorem~2, p.~46]{Halmos-spectral-1951} 
in \cref{sect:more}. 
Various examples
are given
to illustrate
the necessity
of our assumptions.

The notation used
in this 
paper is standard
and mainly follows \cite{bauschke2017convex}. 
We write $A \coloneqq B$ to indicate that 
$A$ is defined to be $B$. 
We set 
$\NN \coloneqq \left\{ 0 ,1, 2, \ldots  \right\}$, 
$\RP \coloneqq \left[0, +\infty \right[$,
$\RPP \coloneqq \sqbo{0, + \infty},$
$\RM \coloneqq \left]-\infty, 0 \right]$,
and $\RMM\coloneqq \sqbo{-\infty, 0}$.
The closed ball
in $\HH$
with center $x \in \HH$
and radius $\rho \in \RPP$
is $\ball{x}{\rho} \coloneqq 
\menge{y \in \HH}{ \norm{y-x} \leq \rho }$.
It
is 
convenient to set 
\begin{empheq}[box = \mybluebox]{equation}\label{q}
\q \coloneqq \tfrac{1}{2} \norm{}^{2}, 
\end{empheq}	
where  $\nabla \q  = \Id $ is
the \emph{identity operator} on $\HH$. 
Let $C$ be a subset of $\HH$.
Then
we denote by
$\closu{C}$ the closure
of $C$ (with respect 
to the norm topology 
on $\HH$), by $d_{C}$
its \emph{distance function,}
by $\pc{C}$
its \emph{polar cone}, i.e.,
$\pc{C}\coloneqq \menge{u \in \HH}{\sup \scal{C}{u} \leq 0},
$ 
and by $C^{\perp}$
its \emph{orthogonal complement.} 
Next,
the 
\emph{indicator}
and 
\emph{support functions}
of $C$
are 
\begin{equation}
\iota_{ C } \colon \HH\to \RRX : 
x\mapsto \begin{cases}
0, & \text{if~}x \in C; \\
+\infty, &\text{otherwise}
\end{cases}
\end{equation}
and 
\begin{equation}
\sigma_{C} \colon \HH \to \RRX : 
u \mapsto \sup \scal{C}{u},
\end{equation}
respectively.
Moreover, if
$C$ is convex, closed, and nonempty,
then the projector
associated with $C$
is denoted by $P_{C}$.
In turn, we set 
\begin{empheq}[box = \mybluebox]{equation}
  \prH \coloneqq
\menge[\Big]{P_{C}}{\HH \supseteq C
\text{ is convex, closed, and nonempty}}.
\end{empheq}
Next, 
the set of 
convex, lower semicontinuous,
and proper
functions from 
$\HH$ to $\RX$
is $\varGamma_{0}\rb{\HH}$. 
The \emph{domain} of 
of a
function $f \colon \HH\to \RRX$
is $\dom{f} \coloneqq \menge{x \in \HH}{f\rb{x } < + \infty} $
with closure $\cdom f$,
its \emph{graph} is denoted by
$\gra f$, 
its \emph{conjugate} 
is denoted by $f^{\ast}$,
and  its \emph{subdifferential}
is denoted by $\partial f$;
furthermore, if $f \in \varGamma_{0}\rb{\HH},$
then we denote its \emph{proximity operator}
by $\prox_{f}$ 
and its \emph{Moreau envelope} by
$\env f$, i.e.,
$\env f \coloneqq f \infconv   \q = f \einfconv \q$,
where $\infconv$ and $\einfconv$ denote
the infimal convolution and 
the exact infimal convolution,
respectively.
Next, let $T \colon \HH\to \HH$.
The \emph{range} of $T$ is $\ran T$ 
with  closure $\cran T$. 
If $T \in \mathscr{B}\rb{\HH}$,
the space of bounded linear operators
on $\HH$,
then its \emph{adjoint}
is denoted by $T^{\ast}$. 
Finally, we adopt
the convention that \emph{empty sums 
are zero}.


\section{Auxiliary results}\label{sect:aux}

In this section,
we provide 
various
results that
will be useful
in the sequel.
Let us start with a
simple identity in $\HH$.

\begin{lemma}
  \label{l:identity-vectors}
  Let $x \in \HH$,
  let $\fa{x_{i}}{i \in I} $
  be a finite family in $\HH$,
 let $\fa{\alpha_{i}}{i \in I} $
 be a family in $\RR$, 
 and set $\alpha \coloneqq \sum_{i \in I}\alpha_{i}$.
  Then the following hold:
  \begin{enumerate}
    \item 
    \label{i:iden-vectors}
  \begin{math}
    \norm*{x-\sum_{i \in I}\alpha_{i}x_{i}}^{2}
    = \rbr{1-\alpha} \norm{x}^{2}
    +\sum_{i \in I}\alpha_{i}\norm{x-x_{i}}^{2}
    +\rbr{\alpha-1}\sum_{i \in I}\alpha_{i} \norm{x_{i}}^{2}
    - \tfrac{1}{2}\sum_{i \in I}\sum_{j \in I}\alpha_{i}\alpha_{j}
    \norm{x_{i}-x_{j}}^{2}.
  \end{math}
\item\label{i:iden-2}
  Suppose that $\rbr{\forall i \in I} \, \alpha_{i}=1$. Then
  $\alpha= \card I$ and
  \begin{equation}
    \label{eq:iden-21}
    \rbr{\alpha-1}\sum_{i \in I} \norm{x_{i}}^{2}
    - \tfrac{1}{2}\sum_{i \in I}\sum_{j \in I} \norm{x_{i}-x_{j}}^{2}
    =  
    \sum_{\substack{\rbr{i,j}\in I\times I \\ i \neq j} }\scal{x_{i}}{x_{j}} 
  \end{equation}
  and 
  \begin{equation}
    \label{eq:iden-22}
    \norm*{x-\sum_{i \in I}x_{i}}^{2}
    =
            \rbr{1-\alpha} \norm{x}^{2}
    +\sum_{i \in I}\norm{x-x_{i}}^{2} 
    +\sum_{\substack{\rbr{i,j}\in I\times I \\ i \neq j} }\scal{x_{i}}{x_{j}} .
  \end{equation}
  \end{enumerate}
\end{lemma}

\begin{proof}
  \cref{i:iden-vectors}:
  Without loss of generality,
  assume that $I = \left\{ 1,\ldots, m \right\} $,
  where $m \in \NN \smallsetminus \left\{ 0 \right\}$.
  Let us proceed by
  induction on $m$.

  \textbf{Base case:} When $m=1$,
  by applying \cite[Corollary~2.15]{bauschke2017convex} to
  $\rbr{x-x_{1},x_{1}} $
  and noticing that $\alpha=\alpha_{1}$, 
  we obtain
  \begin{subequations}
    \begin{align}
      \norm{x-\alpha_{1}x_{1}}^{2}
      & = \norm{\rbr{1-\alpha}x + \alpha\rbr{x-x_{1}}  }^{2}
      \\
      & = \rbr{1-\alpha} \norm{x}^{2} + \alpha\norm{x-x_{1}}^{2}
      - \rbr{1-\alpha} \alpha\norm{x_{1}}^{2} \\
      & = \rbr{1-\alpha} \norm{x}^{2} + \alpha_{1}\norm{x-x_{1}}^{2}
      + \rbr{\alpha-1} \alpha_{1}\norm{x_{1}}^{2}
      - \tfrac{1}{2}\alpha_{1}\alpha_{1}\norm{x_{1}-x_{1}}^{2}.
    \end{align}
  \end{subequations}

  \textbf{Inductive step:} Assume that $m\geq 2$ 
		and that
		the result holds
		for families containing
            $m-1$ or fewer elements.
		Moreover, 
		set $J \coloneqq\{1,\ldots, m-1\} $
        and $\beta \coloneqq \sum_{j \in J}\alpha_{j}$.
        Then, by the base case,
        we have
        \begin{equation}
          \alpha_{m}^{2}\norm{x_{m}}^{2}
          -2\alpha_{m}\scal{x_{m}}{x} =
          \norm{x-\alpha_{m}x_{m}}^{2} - \norm{x}^{2}
          = -\alpha_{m}\norm{x}^{2} 
          +\alpha_{m}\norm{x-x_{m}}^{2}
          +\rbr{\alpha_{m}-1}\alpha_{m}\norm{x_{m}}^{2}.
        \end{equation}
        Hence, 
        since $\beta + \alpha_{m} = \alpha$,
        we infer from the induction
        hypothesis that
        \begin{subequations} 
          \begin{align}
            \norm*{x-\sum_{i \in I} \alpha_{i}x_{i}}^{2} 
            & = \norm*{\rbr[\Bigg]{x-\sum_{j \in J}\alpha_{j}x_{j}} - \alpha_{m}x_{m}}^{2} \\
            & = \norm*{x-\sum_{j \in J}\alpha_{j}x_{j}}^{2}+\alpha_{m}^{2}\norm{x_{m}}^{2}
            - 2\alpha_{m}\scal*{x_{m}}{x-\sum_{j \in J}\alpha_{j}x_{j}} 
            \\
            & = \norm*{x-\sum_{j \in J}\alpha_{j}x_{j}}^{2}+
            \rbr[\Big]{\alpha_{m}^{2}\norm{x_{m}}^{2}
            -2\alpha_{m}\scal{x_{m}}{x}} + 2\alpha_{m}\sum_{j \in
            J}\alpha_{j}\scal{x_{m}}{x_{j}} 
            \\
            & = \norm*{x-\sum_{j \in J}\alpha_{j}x_{j}}^{2}
            -\alpha_{m}\norm{x}^{2}+\alpha_{m}\norm{x-x_{m}}^{2}
            +\rbr{\alpha_{m}-1} \alpha_{m}\norm{x_{m}}^{2}
            \notag \\
            & \qquad \qquad + \alpha_{m}\sum_{j \in J}\alpha_{j}\rbr[\Big]{\norm{x_{m}}^{2}+\norm{x_{j}}^{2}
            -\norm{x_{m}-x_{j}}^{2}} 
            \\
            & = \rbr{1-\beta} \norm{x}^{2}
            + \sum_{j \in J}\alpha_{j}\norm{x-x_{j}}^{2} 
            + \rbr{\beta-1}\sum_{j \in J}\alpha_{j}\norm{x_{j}}^{2}
            - \tfrac{1}{2}\sum_{j \in J}\sum_{k \in
            J}\alpha_{j}\alpha_{k}\norm{x_{j}-x_{k}}^{2} 
            \notag \\
            & \qquad \qquad 
            -\alpha_{m}\norm{x}^{2} +\alpha_{m}\norm{x-x_{m}}^{2} 
            + \alpha_{m}\rbr[\Big]{-1+\alpha_{m}+ \sum_{j \in J}\alpha_{j}} \norm{x_{m}}^{2}
            \notag \\
            & \qquad \qquad + \alpha_{m}\sum_{j \in J}\alpha_{j}\norm{x_{j}}^{2}
            - \sum_{j \in J}\alpha_{m}\alpha_{j}\norm{x_{m}-x_{j}}^{2}
            \\
            & = \rbr{1-\beta-\alpha_{m}} \norm{x}^{2}
            + \sum_{i \in I}\alpha_{i}\norm{x-x_{i}}^{2}
            +\rbr{\beta +\alpha_{m}-1} \sum_{j \in J}\alpha_{j}\norm{x_{j}}^{2} 
            \notag \\
            & \qquad \qquad - \tfrac{1}{2}\sum_{i \in I}\sum_{j \in I}
            \alpha_{i}\alpha_{j}\norm{x_{i}-x_{j}}^{2} 
            +\alpha_{m}\rbr{\alpha-1} \norm{x_{m}}^{2}
            \\
            & = \rbr{1-\alpha} \norm{x}^{2}
            +\sum_{i \in I}\alpha_{i}\norm{x-x_{i}}^{2}
            + \rbr{\alpha-1} \sum_{i \in I}\alpha_{i}\norm{x_{i}}^{2}
            - \tfrac{1}{2}\sum_{i \in I}\sum_{j \in I}
            \alpha_{i}\alpha_{j}\norm{x_{i}-x_{j}}^{2}, 
          \end{align}
        \end{subequations}
         which completes the induction argument.

         \cref{i:iden-2}: 
         Since $\rbr{\forall i \in I} \, \alpha_{i}= 1$,
         we have $\alpha = \card I$,
         and
          thus   
          \begin{subequations} 
           \begin{align}
     & \rbr{\alpha-1}\sum_{i \in I}\norm{x_{i}}^{2}
    - \tfrac{1}{2}\sum_{i \in I}\sum_{j \in I}
    \norm{x_{i}-x_{j}}^{2} \notag \\           
    & \qquad \qquad   =      \sum_{i \in I}\rbr[\Big]{ \rbr{\alpha-1} \norm{x_{i}}^{2}           
            - \tfrac{1}{2} \sum_{j \in I\smallsetminus \left\{ i \right\}}\norm{x_{i}-x_{j}}^{2} }             \\
 & \qquad \qquad   = 
 \sum_{ i \in I}\rbr[\Big]{ \rbr{\alpha-1} \norm{x_{i}}^{2} -
    \tfrac{1}{2}\sum_{j \in I\smallsetminus\left\{ i \right\}}\rbr[\big]{\norm{x_{i}}^{2}  
  + \norm{x_{j}}^{2} - 2\scal{x_{i}}{x_{j}} }   } 
          \\
 & \qquad \qquad   = \sum_{ i \in I} \rbr[\Big]{\rbr{\alpha-1}\norm{x_{i}}^{2}
    - \tfrac{1}{2}\rbr{\alpha-1} \norm{x_{i}}^{2} - 
    \tfrac{1}{2}\sum_{j \in I \smallsetminus\left\{ i \right\}} \norm{x_{j}}^{2}
+ \sum_{j \in I \smallsetminus \left\{ i \right\}} \scal{x_{i}}{x_{j}}  } 
\\
  & \qquad \qquad   = 
     \tfrac{1}{2}\rbr{\alpha-1} \sum_{i \in I}\norm{x_{i}}^{2}
   -  \tfrac{1}{2}\sum_{i \in I} \rbr[\Big]{- \norm{x_{i}}^{2} +
    \sum_{j \in I}\norm{x_{j}}^{2}}  
    + \sum_{\substack{\rbr{i,j}\in I\times I \\ i \neq j} }
    \scal{x_{i}}{x_{j}} 
    \\
   & \qquad \qquad   = 
     \tfrac{1}{2}\rbr{\alpha-1} \sum_{i \in I}\norm{x_{i}}^{2}
    +\tfrac{1}{2}\sum_{i \in I} \norm{x_{i}}^{2} 
    - \tfrac{\alpha}{2}\sum_{i \in I}\norm{x_{i}}^{2}   
+\sum_{\substack{\rbr{i,j}\in I\times I \\ i \neq j} }
    \scal{x_{i}}{x_{j}}
    \\
     & \qquad \qquad   = 
    \sum_{\substack{\rbr{i,j}\in I\times I \\ i \neq j} }\scal{x_{i}}{x_{j}} ,
           \end{align}
          \end{subequations}
        and hence \cref{eq:iden-21} holds.
        Consequently, \cref{eq:iden-22}
        follows from \cref{i:iden-vectors} and \cref{eq:iden-21}.
\end{proof}

We shall need
the following identities involving
convex cones.

\begin{lemma}\label{l:cone-identities}
	Let $K$ and $S$ be nonempty
	closed convex cones in $\HH$.
	Then the following hold: 
	\begin{enumerate}
		\item\label{i:cone-identities1} $\rb{\forall x \in \HH}
              \, 
		\norm{P_{K}x}^{2} = \scal{x}{P_{K}x}$.
		\item\label{i:cone-identities2} 
		\begin{math}
			\rb{\forall x  \in \HH}
                  \,
			\scal{P_{\pc{K} } x}{  P_{\pc{S} }x}
			+ \norm{P_{K}x}^{2} + \norm{P_{S}x}^{2}
			= \norm{x}^{2} + \scal{P_{K}x}{P_{S}x}.		
		\end{math}
	\end{enumerate}
\end{lemma}

\begin{proof}
	Take $x \in \HH$. 
	\cref{i:cone-identities1}: 
	We derive from \cite[Theorem~6.30(i)\&(ii)]{bauschke2017convex}
	that 
	\begin{math}
		\norm{P_{K}x}^{2} = \scal{P_{K}x}{P_{K}x}
		= \scal{ x- P_{\pc{K}}x}{P_{K}x} = \scal{x}{P_{K}x},
	\end{math}
	as claimed.
	\cref{i:cone-identities2}: 
	The Moreau conical decomposition
      (\cite{Moreau-cone-1962})
	and \cref{i:cone-identities1}
	give  
	\begin{subequations}
		\begin{align}
			\scal{P_{\pc{K} } x}{  P_{\pc{S} }x} 
			& = \scal{x - P_{K}x }{ x - P_{S}x} \\ 
			& = \norm{x}^{2} 
			- \scal{x}{P_{S}x} - \scal{x}{P_{K}x}
			+	\scal{P_{K}x}{ P_{S}x} \\ 
			& = \norm{x}^{2}
			- \norm{P_{S}x}^{2}
			- \norm{P_{K}x}^{2} 
			+ \scal{P_{K}x}{ P_{S}x},
		\end{align}
	\end{subequations}
	and the assertion follows.
\end{proof}

\begin{fact}\label{eg:PC-mono}
	Let $C$ be a nonempty
	closed convex subset of $\HH$.
	Then the following hold:
	\begin{enumerate}
		\item\label{i:Pc-maxmono} $P_{C}$ is maximally monotone.
		\item\label{i:PC-3*mono} $P_{C}$
		is $3^{\ast}$ monotone\footnote{A monotone
			operator $A \colon \HH \to 2^{\HH}$
			is \emph{$3^{\ast}$ monotone}
			if 
			\begin{math}
				\rb{\forall \rb{x,u} \in \dom{A} \times \ran{A}}
                        \,
				\inf_{\rb{y,v} \in \gra{A} } \scal{x-y}{u-v} > -\infty.
			\end{math}}.
	\end{enumerate}
\end{fact}

\begin{proof}
	\cref{i:Pc-maxmono}: See 
	\cite[Example~20.32]{bauschke2017convex}.
	\cref{i:PC-3*mono}:  Because $P_{C}$
	is firmly nonexpansive 
	by \cite[Proposition~4.16]{bauschke2017convex},
	the conclusion follows
	from \cite[Example~25.20(ii)]{bauschke2017convex}.
\end{proof}

In the finite-dimensional case,
\cref{p:sumCi}\cref{i:set-sum1}
can also be deduced from 
\cite[Theorem~3.15]{Bauschke-Wang-Moffat-2013}. 
Furthermore, 
let us point out that 
\cref{p:sumCi}\cref{i:set-sum2} 
generalizes 
Zarantonello's
\cite[Theorem~5.4]{zarontello1971projections-partI}.

\begin{proposition}\label{p:sumCi}
  Let $\fa{C_{i}}{i \in I} $
	be a finite family of nonempty
	closed convex subsets of 
    $\HH$,
     let $\fa{\alpha_{i}}{i \in I} $ be
     a family in $\RR$,
     and set $\alpha\coloneqq \sum_{i \in I}\alpha_{i}$.
	Then the following hold: 	
	\begin{enumerate}
		\item\label{i:sum-many1} For every $x \in \HH$,
		\begin{align}
              & 	\q\rbr[\bigg]{ x - \sum_{i \in I} \alpha_{i}P_{C_{i}}x }
        = \tfrac{1}{2}\sum_{i \in I}\alpha_{i}d_{ C_{i} }^{2}\rb{x}
			- \rb{\alpha-1}\q\rb{x} 
            \notag  \\
         &   \qquad \qquad 
\qquad \qquad \qquad 
\qquad 
+    \rbr{\alpha-1} \sum_{i \in I}\alpha_{i}\q\rbr{P_{C_{i}}x} 
      - \tfrac{1}{2}\sum_{i \in I}\sum_{j \in I}\alpha_{i}\alpha_{j}
            \q\rbr{P_{C_{i}}x - P_{C_{j}}x}            .
		\end{align}
		\item\label{i:set-sum1}
              Suppose that $\rbr{\forall i \in I}\, \alpha_{i}
              \geq 0$.
          Then
		\begin{math}
          \cran{ \sum_{ i \in I} \alpha_{i}P_{C_{i}} } 
          = \closu{\sum_{ i \in I}\alpha_{i} C_{i} }.
		\end{math}
		\item \label{i:set-sum2}
		Suppose that
            $\rbr{\forall i \in I} \, \alpha_{i} \geq  0$
        and that 
		there exists a closed convex
		set $C$ such that 
        $\sum_{ i \in I} \alpha_{i}P_{C_{i}} = P_{C}$.
        Then $\sum_{ i \in I} \alpha_{i}C_{i}$
        is closed and $C= \sum_{ i \in I} \alpha_{i}C_{i}$.
	\end{enumerate}
\end{proposition}

\begin{proof}
	
	\cref{i:sum-many1}: 
    Let $x $ be in $\HH$.
    Apply \cref{l:identity-vectors}\cref{i:iden-vectors}
    to $\rbr{x,\fa{P_{C_{i}}x}{i \in I} , \fa{\alpha_{i}}{i \in I}  } $
    and notice that $\rbr{\forall i \in I} \,\norm{x-P_{C_{i}}x} =d_{C_{i}}\rbr{x} $.

	\cref{i:set-sum1}: 
    Because the operators $\fa{P_{C_{i}}}{i \in I} $
  are $3^{\ast}$ monotone
  by \cref{eg:PC-mono}\cref{i:PC-3*mono}
  and because $\rbr{\forall i \in I} \, \dom P_{C_{i}} = \HH$,
  we derive from \cite[Lemma~3.1(ii)]{bauschke-walaa-dec2017}
  that
  \begin{math}
    \cran \sum_{i \in I}\alpha_{i}P_{C_{i}}
    = \closu{\sum_{i \in I}\alpha_{i}\ran P_{C_{i}} }
    = \closu{\sum_{i \in I}\alpha_{i}C_{i} },
  \end{math}
  as desired. 

	\cref{i:set-sum2}: 
	It follows from \cref{i:set-sum1}
	and our assumption that 
	\begin{equation}
      \sum_{ i \in I} \alpha_{i}C_{i}
      \subseteq \closu{\sum_{ i \in I} \alpha_{i }C_{i}}
      = \cran \sum_{ i \in I}\alpha_{i} P_{C_{i}}
		=\cran P_{C}
		= C
		= \ran P_{C}
        = \ran \sum_{ i \in I} \alpha_{i}P_{C_{i}}
        \subseteq \sum_{ i \in I} \alpha_{i}C_{i}.
	\end{equation}
	Thus, we conclude that 
    $\sum_{ i \in I} \alpha_{i}C_{i} = C$
    and that $\sum_{ i \in I}\alpha_{i} C_{i}$ 
    is closed.
\end{proof}

  \begin{remark}\label{rem:counter-sum}
    \cref{p:sumCi}\cref{i:set-sum1}\&\cref{i:set-sum2}
    may fail if 
    $\rbr{\exists i \in I}~\alpha_{i} <0$.
    Indeed,
    in the setting of \cref{p:sumCi},
    suppose that 
    $I = \left\{ 1,2 \right\} $,
    that $C_{1}=C_{2}$,
    and that $\alpha_{1}=-\alpha_{2}=1$.
    Then
  $\alpha_{1}P_{C_{1}}+\alpha_{2}P_{C_{2}} = 0 
  = P_{\left\{ 0 \right\}}$,
  but $\alpha_{1}C_{1} + \alpha_{2}C_{2}
  =C_{1}- C_{1} \neq \left\{ 0 \right\}$ 
  if $C_{1}$ is not a singleton.
  \end{remark}

\begin{lemma}\label{l:ph-mono-T0}
	Let $T \colon \HH \to \HH$
	be monotone and 
	 positively homogeneous\footnote{A mapping
		$F\colon \HH\to \HH$ is \emph{positively homogeneous}
            if $\rb{\forall x \in \HH}\rb{\forall \lambda \in \RPP}\, F\rb{\lambda x}
		=\lambda Fx$. Note that we
		do not require $F0 =0$.}.
	Then $T0 = 0$.
\end{lemma}

\begin{proof}
	Set $x \coloneqq T0$. Then, 
	since $T$
	is monotone and positively homogeneous,
	\begin{math}
        \rb{\forall \lambda \in \RPP}\, 
	0 \leq \scal{T\rb{\lambda x} - T0 }{\lambda x - 0}
	= \lambda \scal{\lambda Tx - T0}{x},
	\end{math}
	from which we infer that 
	\begin{math}
	\rb{\forall \lambda \in \RPP}
      \,
	\scal{T0}{x} \leq \lambda \scal{Tx}{x}.
	\end{math}
	Hence, letting $\lambda \downarrow 0$ yields 
	$\norm{T0}^{2} = \scal{T0}{x} \leq 0$. 
	Consequently, $T0 =0$.
\end{proof}

The following is 
a variant of 
\cite[Lemma~6.1]{zarontello1971projectionsII}.
We provide a proof for completeness.

\begin{lemma}\label{l:derivative}
	Let $f \colon \HH \to \RR$ 
	be G\^{a}teaux differentiable
	on $\HH$,
	and suppose that 
	$\nabla f$
	is monotone and 
	positively homogeneous.
	Then 
	\begin{equation}
	\rb{\forall x \in \HH}
	\quad 
	f\rb{x} = \tfrac{1}{2}\scal{x}{\nabla f\rb{x}} + f\rb{0}.
	\end{equation}
\end{lemma}

\begin{proof}
	By \cref{l:ph-mono-T0} and our assumption,
	$\nabla f\rb{0} =0$, which implies that 
	\begin{equation}\label{e:ph-grad}
	\rb{\forall x \in \HH}\rb{\forall \lambda \in \RP}
	\quad 
	\nabla f\rb{\lambda x } = \lambda \nabla f\rb{x}.
	\end{equation}
	Now fix $x \in \HH$, and 
	set $\phi \colon \RR \to \RR :
	t\mapsto f\rb{tx} $.
	Then, since $f$ is G\^{a}teaux
	differentiable, we 
	see that  
	\begin{subequations}
		\begin{align} 
		\rb{\forall t \in \RR}
		\quad 
		\lim_{0\neq \alpha  \to 0}\frac{\phi\rb{t+\alpha} - \phi\rb{t}}{\alpha}
		& =  \lim_{0\neq \alpha \to 0}
		\frac{f\rb{ \rb{t+\alpha}x } - f\rb{tx} }{\alpha} \\ 
		& = \lim_{0\neq \alpha \to 0} 
		\frac{ f\rb{tx + \alpha x} - f\rb{tx} }{\alpha}  \\ 
		&  = \scal{x}{ \nabla f \rb{tx} }. \label{e:d-phi}
		\end{align}
	\end{subequations}
	Hence, $\phi$
	is differentiable on $\RR$
	and, in view of \cref{e:ph-grad}\&\cref{e:d-phi}, 
	\begin{math}
        \rb{\forall t \in \RP } \,  \phi'\rb{t} = t\scal{x}{\nabla f\rb{x}}.
	\end{math}
	Consequently,
	\begin{equation}
	f\rb{x} - f\rb{0} = \phi\rb{1} - \phi\rb{0}
	= \int_{0}^{1}\phi'\rb{t}dt 
      = \int_{0}^{1}t\scal{x}{\nabla f\rb{x}}dt
	= \tfrac{1}{2}\scal{x}{\nabla f \rb{x}},
	\end{equation}
	from which we obtain the conclusion.
\end{proof}

Recall 
from \cite[pp.~89{\textendash}90]{Boris-I}
that, if $f \colon \HH \to \left]-\infty,+\infty\right]$,
then the \emph{Fr\'{e}chet subdifferential}
of $f$ is 
\begin{equation}
\hat{\partial}f \colon 
\HH\to 2^{\HH} : 
x \mapsto 
\menge*{u \in \HH}
{\varliminf_{\substack{ x \neq y \to x  }}  
	\frac{f\rb{y} - f\rb{x} - \scal{u}{y-x} }{\norm{y-x}} \geq 0 }.
\end{equation}

\begin{lemma}\label{l:F-subdif}
	Let $f \colon \HH \to \RR $
	and $z \in \HH$. Suppose that
	$f$ is Fr\'{e}chet differentiable
	on $\HH$. Then 
	\begin{equation}\label{e:F-sub-sumrule}
	\rb{\forall \varepsilon \in \RPP}
	\quad 
	\hat{\partial}\rb{f + \varepsilon  d_{ \left\{z\right\} }}\rb{z}
	= \nabla f\rb{z} + \varepsilon \ball{0}{1}.
	\end{equation}
\end{lemma}

\begin{proof}
	Fix $\varepsilon \in \RPP.$
	Since $f$ is Fr\'{e}chet
	differentiable 
	and $d_{\{z\}}$
	is convex, 
	we derive from 
	\cite[Proposition~1.107 and Theorem~1.93]{Boris-I}
	that 
	\begin{math}
	\hat{\partial}\rb{g + \varepsilon  d_{\{z\}}}\rb{z} 
	= \nabla f\rb{z} + \hat{\partial}\rb{\varepsilon d_{ \left\{z\right\} }}\rb{z}
	= \nabla f \rb{z} + \partial \rb{\varepsilon d_{\left\{z\right\}}}\rb{z}
	= \nabla f \rb{z} + \varepsilon \partial d_{\left\{z\right\}}\rb{z}.
	\end{math}
	Hence, in view of \cite[Example~16.62]{bauschke2017convex} 
	(applied to $C = \left\{z\right\}$),
	\cref{e:F-sub-sumrule} follows.
\end{proof}


\section{Main results}\label{sect:main}

  \begin{theorem}[Characterization theorem
    for proximity operators]\label{t:cha-prox}
  \label{t:prox-charac}
  Let $\varphi \in \varGamma_{0}\rb{\HH}  $,
  let $T \colon \HH \to \HH$,
  and set $f \coloneqq \varphi \circ T
  + \q \circ \rb{ {\Id} - T } $.
  Then the following
  are equivalent: 
  \begin{enumerate}
    \item\label{i:prox1} $T = \prox_{\varphi}$.
    \item\label{i:prox2} $T$ is monotone,
       $\gra\rbr{\varphi + \iota_{\ran T}} $
      is a dense subset of $\gra \varphi$,
      and $f$ is G\^ateaux differentiable
      on $\HH$ with $\nabla f = \Id - T$.
  \end{enumerate}
  Furthermore, if \cref{i:prox1} 
  or \cref{i:prox2} holds,
  then $f=\env \varphi$ 
  and $f$ is Fr\'{e}chet differentiable 
  on $\HH$.
\end{theorem}

\begin{proof}
  ``\cref{i:prox1}\ensuremath{\implies}\cref{i:prox2}'':
  First, by \cite[Example~20.30]{bauschke2017convex},
  $T = \prox_{\varphi}$ is monotone.
  Next, since $\varphi \in \varGamma_{0}\rbr{\HH} $
  and $T = \prox_{\varphi}$,
  \cite[Eq.~(24.3)]{bauschke2017convex} gives
  $\ran T = \dom \partial\varphi$,
  and hence, 
  according to \cite[Proposition~16.38]{bauschke2017convex},
  it follows that 
$\gra\rbr{\varphi + \iota_{\ran T}} $
is a dense subset of $\gra \varphi$.
Finally, in view of \cite[Remark~12.24]{bauschke2017convex},
we see that $f = \varphi \circ T + \q \circ \rbr{\Id - T} 
= \varphi \circ {\prox_{\varphi}} + \q \circ \rbr{\Id - {\prox_{\varphi}}} 
= \env \varphi$,
and   \cite[Proposition~12.30]{bauschke2017convex}
thus entails
that $f$ is Fr\'echet (thus G\^ateaux) differentiable 
on $\HH$ with $\nabla f = \Id - {\prox_{\varphi}} = \Id - T$.

``\cref{i:prox1}\ensuremath{\impliedby}\cref{i:prox2}'':
	Set $g \coloneqq \q - f$.
	Then, 
	on the one hand,
	because $\q$ and $f$
	are G\^{a}teaux differentiable,
	so is $g$. On the other hand, 
	since 
	$\nabla \q = \Id $ 
	and $\nabla f = \Id - T$,
	we infer that 
	\begin{math}
	\nabla g = \nabla \rb{\q - f} = \nabla \q - \nabla f = T,
	\end{math} 
	which is monotone by assumption.
	Altogether, 
	\cite[Proposition~17.7]{bauschke2017convex}
	yields the convexity of $g$. 
	Therefore,
	since $g$ is G\^{a}teaux differentiable
	on $\HH$, 
	it follows from 
      \cite[Proposition~17.48(i)]{bauschke2017convex}
	that 
	$g$ is lower semicontinuous on $\HH$. 
	To sum up, we have shown that 
	\begin{equation}\label{e:info-g}
	\text{$g = \q -f $ belongs to $\lscc\rb{\HH}$ 
		and is G\^{a}teaux differentiable on $\HH$
		with $\nabla g= T$.}
	\end{equation}
    Moreover, $\cref{e:info-g}$ and 
    \cite[Corollary~13.38]{bauschke2017convex}
    yield 
    \begin{equation}
      \label{eq:infor-gconj}
      g^{\ast} \in \varGamma_{0}\rbr{\HH} .
    \end{equation}
	In turn, set
	\begin{math}
      h \coloneqq g^{\ast} -\q.
	\end{math}
    Let us now establish that 
    \begin{equation}
      \label{eq:todo-prox}
      h = \varphi \text{~on~} \ran T.
    \end{equation}
	Towards this goal, fix $u \in \ran{T}$,
      say $u = Tx \overset{\cref{e:info-g}}{=}
      \nabla g\rb{x}$, where $x \in \HH$.
	Then \cref{e:info-g}, 
	\cite[Proposition~17.35]{bauschke2017convex},
	and the very definitions of $g$
	and $f$ 
	assert that 
	\begin{subequations} 
		\begin{align}
		h\rb{u} 
		= g^{\ast}\rb{u} - \q\rb{u}
		& = g^{\ast}\rb{\nabla g \rb{x}} - \q\rb{Tx} \\
		& = \scal{x}{\nabla g\rb{x}} - g\rb{x} - \q\rb{Tx} \\
		& = \scal{x}{Tx} - \q\rb{x} + f\rb{x} - \q\rb{Tx} \\
		& = \scal{x}{Tx} - \tfrac{1}{2}\norm{x}^{2} + \tfrac{1}{2}\norm{x-Tx}^{2}
        + \varphi\rbr{Tx} 
		- \tfrac{1}{2}\norm{Tx}^{2} \\
        & = \varphi\rbr{Tx} \\
        & = \varphi\rbr{u}.
		\end{align}
	\end{subequations}
    Hence, \cref{eq:todo-prox} holds. 
    Next, fix $v \in \dom h $, and we shall
    prove that $\varphi \rbr{v} \leq h\rbr{v} $.
    Indeed, 
    on the one hand,
    because $h = g^{\ast} - \q$
    and $\dom \q = \HH$,
    we have $\dom h = \dom g^{\ast}$.
	On the other hand,
    due to
    \cref{e:info-g} and
    \cite[Corollary~16.30]{bauschke2017convex},
	\begin{math}
	\dom \partial g^{\ast}
	= \dom \rb{\partial g}^{-1}
    = \ran \partial g,
	\end{math}
    and since
	$\ran \partial g = \ran \nabla g = \ran T $
      thanks to 
	\cref{e:info-g}
	and 
	\cite[Proposition~17.31(i)]{bauschke2017convex},
    we deduce that 
    $\dom \partial g^{\ast} = \ran T$.
    Altogether, because
    $v \in \dom h = \dom g^{\ast}$,
    \cref{eq:infor-gconj} and 
    \cite[Proposition~16.38]{bauschke2017convex}
    ensures the existence of
    a sequence $\fa{v_{n}}{n \in \NN} $
    in $\dom \partial g^{\ast} = \ran T$
    such that $v_{n} \to v$
    and $g^{\ast}\rbr{v_{n}}  \to g^{\ast}\rbr{v} $.
    Therefore,
    by the definition of $h$,
    we get 
    $h\rbr{v_{n}}  = g^{\ast}\rbr{v_{n}} - \q\rbr{v_{n}} 
    \to g^{\ast}\rbr{v} -\q\rbr{v}  = h\rbr{v} $.
    However, because 
    $\set{v_{n}}{n \in \NN} \subset \ran T$ and 
    $v_{n} \to v$,
    the lower semicontinuity
    of $\varphi$ and 
    \cref{eq:todo-prox} imply that 
    \begin{math}
      h\rbr{v} = \lim h\rbr{v_{n}} 
      = \lim \varphi \rbr{v_{n}} 
      \geq \varphi \rbr{v}.
    \end{math}
    Hence, 
    we have established that 
    \begin{equation}
      \label{eq:hphi-geq}
      \rbr{\forall v \in \dom h} \quad
      h\rbr{v} \geq \varphi \rbr{v} .
    \end{equation}
    Next, let us show that 
    \begin{equation}
      \label{eq:hphi-leq}
      \rbr{\forall w \in \dom \varphi} \quad
      h\rbr{w} \leq \varphi\rbr{w} .
    \end{equation}
    To this end,
    let $w \in \dom \varphi$.
    Then, since $\gra\rbr{\varphi +\iota_{\ran T}} $
    is a dense subset of $\gra \varphi$
    by assumption,
    there exists a sequence $\fa{w_{n}}{n \in \NN} $
    in $\ran T$ such that $w_{n}\to w$
    and $\varphi\rbr{w_{n}} \to \varphi\rbr{w} $.
    In turn, since $h$ is
    lower semicontinuous by \cref{eq:infor-gconj},
    we infer from \cref{eq:todo-prox}
    that 
    \begin{math}
      \varphi\rbr{w} =
      \lim \varphi\rbr{w_{n}} 
      = \lim h\rbr{w_{n}} 
      \geq h\rbr{w} ,
    \end{math}
    from which \cref{eq:hphi-leq} follows.
    Consequently, combining
    \cref{eq:hphi-geq} and \cref{eq:hphi-leq}
    yields $h=\varphi$.
    Finally, 
    since $\varphi \in \varGamma_{0}\rbr{\HH}$,
    it follows from \cref{e:info-g},
    the definition of $h$, 
    the Fenchel--Moreau theorem, 
    and \cite[Proposition~24.4]{bauschke2017convex}
    that 
    \begin{math}
      \prox_{\varphi} = 
      \nabla \rbr{\varphi + \q}^{\ast}
      =\nabla \rbr{h + \q}^{\ast}
      = \nabla g^{\ast\ast}
      =\nabla g
      = T,
    \end{math}
    as desired.
\end{proof}

In \cite{zarontello1971projections-partI},
Zarantonello provided
a necessary and sufficient 
condition in terms of
a differential equation
for an operator on $\HH$
to be a projector. 
The proof there, however,
is not within the 
scope of Convex Analysis.  
He also
conjectured 
(see the paragraph
after \cite[Corollary~2, p.~306]{zarontello1971projections-partI})
that 
the Fr\'{e}chet differentiability
of the operator $P$ 
in \cite[Theorem~4.1]{zarontello1971projections-partI}
can be replaced by the G\^{a}teaux one.
By assuming 
the monotonicity of $P$ instead
of the 
Lipschitz continuity, 
we  provide below an
affirmative answer.
The next result,
which plays  a crucial role
in determining
whether a sum of
projectors is a projector 
(see \cref{t:sum-many} below), 
is 
a variant of \cite[Theorem~4.1]{zarontello1971projections-partI}
with a proof
rooted in Convex Analysis.

\begin{theorem}[Characterization
  theorem for projectors]\label{t:main}
	Let $T\colon \HH \to \HH,$
	and set $f \coloneqq \q \circ \rb{\Id -T}$.
	Then the following are equivalent:
	\begin{enumerate}
		\item\label{i:main1} 
              $T \in \prH$.
		\item\label{i:main2}
		$T$ is monotone, $f$ is 
		G\^{a}teaux differentiable
		on $\HH$, 
		and $\nabla f = \Id - T$.
	\end{enumerate}
	If \cref{i:main1} or \cref{i:main2} holds,
	then 
	$\ran{T}$ is closed and convex, 
	$T = P_{\ran{T}}$, and $f = \rb{1/2}d_{\ran{T}}^{2}$
	is Fr\'{e}chet differentiable on $\HH$.
\end{theorem}

\begin{proof}
  Set $\varphi \coloneqq \iota_{\cran T}$.

	``\cref{i:main1}\ensuremath{\implies}\cref{i:main2}'': 
      Suppose that $T= P_{C}$,
      where $C$ is convex, closed,
      and nonempty.
      Then
	 clearly $   \ran{T} =
	\ran{P_{C}} = C$
	is closed and convex. 
      This implies that $\varphi = \iota_{\ran T} \in 
      \varGamma_{0}\rbr{\HH}$
      and that 
      $T = P_{\ran T} = \prox_{\iota_{\ran T}} 
      = \prox_{\varphi}$.
      In turn, 
      because $f = \varphi \circ T + \q \circ \rbr{
      \Id - T}$
      by the definition of $\varphi$, 
      we infer from \cref{t:prox-charac}
      that \cref{i:main2} holds 
      and, moreover, $f= \env \varphi 
      =\env \iota_{\ran T} = \rbr{1/2}d_{\ran T}^{2}$
    is Fr\'{e}chet differentiable on $\HH$.

    ``\cref{i:main1}\ensuremath{\impliedby}\cref{i:main2}'':
    We first show that 
    $\cran T$ is convex.
    Indeed,
    by our assumption,
    $\q - f$ is G\^ateaux  differentiable
    on $\HH$ with 
    \begin{equation}
      \label{eq:diff-qf}
    \nabla \rbr{\q - f}  = \nabla \q - \nabla f 
    = \Id - \rbr{\Id - T}  = T.
  \end{equation}
    Thus, since $T$ is monotone,
    \cite[Proposition~17.7]{bauschke2017convex}
    ensures that $\q - f$ is convex,
    and thus, 
    the G\^ateaux
differentiability of $\q -f $ 
and 
\cite[Proposition~17.48(i)]{bauschke2017convex}
    imply that $\q - f \in \varGamma_{0}\rbr{\HH} $.
    Hence, due to \cref{eq:diff-qf}
    and \cite[Proposition~17.31(i)]{bauschke2017convex},
     Moreau's theorem \cite{Moreau-convex-1966}
    asserts that $T =\nabla \rbr{\q -f} $
    is maximally monotone.
    Consequently, 
    \cite[Corollary~21.14]{bauschke2017convex}
    yields the convexity of $\cran T$,
    as claimed.
    In turn,
    on the one hand, 
    this implies that $\varphi =\iota_{\cran T} \in 
    \varGamma_{0}\rbr{\HH}$. 
    On the other hand, 
    we deduce 
    from the definition of $\varphi$
    that 
    $ f = \varphi \circ T + \q \circ \rbr{\Id - T} $
    and 
    $\gra\rbr{\varphi + \iota_{\ran T}}  = \gra\rbr{\iota_{\cran T \cap \ran T}} 
    = \gra \iota_{\ran T} = \ran T \times \left\{ 0 \right\}$
    is a dense subset of 
    $\cran T \times \left\{ 0 \right\}
    = \gra \iota_{\cran T}
    = \gra \varphi$.
    Thus, the implication 
    ``\cref{i:prox2}\ensuremath{\implies}\cref{i:prox1}''
    of \cref{t:prox-charac}
    and our assumption
    guarantee that 
    $T = \prox_{\varphi} = \prox_{\iota_{\cran T}} = P_{\cran T}$,
  which completes the proof.
\end{proof}

\begin{remark}
	Consider the  
	implication 
	``\cref{i:main2}\ensuremath{\implies}\cref{i:main1}''
	of \cref{t:main}. 
	If we merely assume that 
	$T$ is defined on a proper
	open subset $D$ of $\HH$,
	then, although
	there may exist 
	a closed set $C$
	such that $T$ is the
	restriction to $D$ of the 
	projector onto $C$,
	the set $C$ may fail to be convex.
	An example can be constructed as follows.  
	Suppose that $\HH\neq \left\{0\right\}$, 
	and set  
	\begin{equation}
	T \colon \HH\smallsetminus \left\{ 0 \right\} \to \HH: 
	x \mapsto \frac{x}{\norm{x}}, 
	\quad 
	f \coloneqq \q \circ \rb{\Id - T},
	\quad 
	\text{and~} C \coloneqq \menge{x \in \HH}{ \norm{x} = 1}, 
	\end{equation}
	i.e., $C$ is the unit sphere of $\HH$.
	Then clearly $C$ is a closed nonconvex set and  $T$
	is the restriction to $\HH\smallsetminus \left\{0\right\}$ 
	of the set-valued projector 
	$P_{C}$.
	Thus, in the light of 
	\cite[Example~20.12]{bauschke2017convex}, $T$ is monotone. 
	Next, since 
	\begin{math}
	\rb{\forall x \in \HH \smallsetminus \left\{0\right\}}
      \,
	f\rb{x} = \rb{1/2}\norm{\rb{1-1/\norm{x}}x }^{2}
	= \rb{1/2}\rb{\norm{x} -1}^{2}
	= \q \rb{x} - \norm{x} +1/2,
	\end{math}
	we infer that $f$ is Fr\'{e}chet
	differentiable on $\HH \smallsetminus \left\{0\right\}$
	and 
	\begin{equation}
	\rb{\forall x \in \HH\smallsetminus\left\{0\right\}}
	\quad 
	\nabla f \rb{x} = x - \frac{x}{\norm{x}} = x - Tx.
	\end{equation}
\end{remark}

\begin{openprob}\label{rem:open-prob}
	We do not know whether the monotonicity
	of $T$ can be omitted in
	\cref{t:main}. 
	Nevertheless, 
    on
    the one hand,
	the following
	remark might 
	be useful in finding
	counterexamples
	if one thinks 
	the answer is negative;
    on the other hand,
  \cref{p:FixT-nonempty}
provides information on
the set $\Fix T$ in the absence of 
monotonicity.
\end{openprob}

\begin{remark}[{\cite{Cosner}}] 
\label{rem:R2}
	Consider the setting of \cref{t:main}
	and suppose that $\HH = \RR^{2}$. 
	Set $F \coloneqq  \Id -T$
      and $\rb{\forall \rb{x,y } \in \HH}\,
	F\rb{x,y} \coloneqq \rb{F_{1}\rb{x,y},F_{2}\rb{x,y}}$. 
	Now assume that $f$ is
	Fr\'{e}chet differentiable on $\HH$
	with $\nabla f = \Id -T = F$;
	in addition, suppose that 
      $F_{1}$ and $F_{2}$ are continuously
      differentiable. 
	Then,
	since $\rb{F_{1},F_{2}} = \nabla f$,
	it follows that $\partial f / \partial x = F_{1}$
	and that $\partial f / \partial y = F_{2}$.
	Hence, due to Schwarz's theorem 
	(see, e.g., \cite[Theorem~4.1]{coleman2012calculus}), 
	\begin{equation}
		\frac{ \partial F_{1} }{ \partial y  }
		 = \frac{\partial^{2}f }{ \partial y \partial x }
		 = \frac{\partial^{2}f }{ \partial x \partial y }
		 =\frac{ \partial F_{2} }{ \partial x  }.
	\end{equation}
	However, because $\nabla f =  F$, 
	a direct computation gives 
	\begin{equation}\label{e:pde-system}
	\left\{ 	\begin{array}{l}
	\displaystyle 	
	F_{1}\rb{x,y} = F_{1}\rb{x,y} \frac{\partial F_{1} }{\partial x}\rb{x,y}	
	+ F_{2}\rb{x,y} \frac{\partial F_{2} }{\partial x}\rb{x,y} 
	= F_{1}\rb{x,y} \frac{\partial F_{1} }{\partial x}\rb{x,y}	
	+ F_{2}\rb{x,y} \frac{\partial F_{1} }{\partial y}\rb{x,y}
	\medskip \\
	\displaystyle 	F_{2}\rb{x,y} 
	= F_{1}\rb{x,y} \frac{\partial F_{1} }{\partial y}\rb{x,y}	
	+ F_{2}\rb{x,y} \frac{\partial F_{2} }{\partial y}\rb{x,y}
	= F_{1}\rb{x,y} \frac{\partial F_{2} }{\partial x}\rb{x,y}	
	+ F_{2}\rb{x,y} \frac{\partial F_{2} }{\partial y}\rb{x,y}.
	\end{array}
	\right. 
	\end{equation}
	In the first equation of \cref{e:pde-system}, 
	one can try to solve for $F_{1}$
	in term of $F_{2}$, and vise versa by
	using the second one.
	This 
	approach recovers
	projectors onto linear subspaces
	of $\RR^{2}$
	and 
	might suggest 
	a nonmonotone solution of
	the equation
	$\nabla f = \Id - T.$
	In addition,  
	it is worth noticing
	that the function 
	$g \colon x \mapsto \norm{x-Tx}$
	satisfies 
	the eikonal equation 
	(see, e.g., \cite{Bardi-Dolcetta-1997}),
	i.e., 
	\begin{math}
		\rb{\forall x \in \HH \smallsetminus 
            \closu{\Fix{T}}}\,
		\norm{\nabla g\rb{x}} =1.
	\end{math}
	This might give us
	some
	insights
	into 
	\cref{rem:open-prob}.
\end{remark}

\begin{proposition}\label{p:FixT-nonempty}
	Let $T\colon \HH\to \HH$,
	and set $f \coloneqq \q \circ \rb{\Id -T}$. 
	Suppose that 
	$f$ is Fr\'{e}chet differentiable
	on $\HH$\
	with $\nabla f = \Id -T$.
	Then $\Fix T \neq \varnothing.$
\end{proposition}

\begin{proof}
	Let us proceed by contradiction
	and therefore assume that 
	$\Fix T = \varnothing.$
      Then clearly $\rb{\forall x \in \HH}\,f\rb{x} > 0$.
	Hence, 
	because $f \colon \HH \to \RPP$
	is Fr\'{e}chet differentiable
	with $\nabla f = \Id - T$ 
	and $\sqrt{\, \cdot \,} \colon \RPP\to \RR$
	is Fr\'{e}chet differentiable, 
	we
	deduce 
	from \cite[Theorem~5.1.11(b)]{DenkowskiMigorskiPapageorgiou-nonlinear}
	that $g \coloneqq \sqrt{\, \cdot \,} \circ \rb{2f} $
	is Fr\'{e}chet differentiable
	on $\HH$
	(thus continuous)
	and 
	\begin{equation}\label{e:Ekeland-grad-g}
	\rb{\forall x \in \HH}
	\quad 
	\nabla g\rb{x} = 
	\frac{2\nabla f\rb{x}}{2\sqrt{2f\rb{x}}}
	=	\frac{x - Tx}{\norm{x - Tx}}.
	\end{equation}
	Now let $\varepsilon \in \sqbo{0,1 }$. 
	Since $g$ is bounded below
	and continuous, 
	Ekeland's variational 
	principle
	(see, e.g., \cite[Theorem~1.46(iii)]{bauschke2017convex})
	applied to $g$ and 
	$\rb{\alpha ,\beta} = \rb{\varepsilon^{2},\varepsilon}$
	yields the existence 
	of $z \in \HH$
      such that $\rb{\forall x \in \HH\smallsetminus\{z\}}\,
	g\rb{z} + \varepsilon d_{\left\{z\right\} }\rb{z} = g\rb{z} <   g\rb{x} + \varepsilon d_{\left\{z\right\}}\rb{x} $.
	This guarantees that 
	$z$ is the unique minimizer
	of $g + \varepsilon d_{\left\{z\right\}}$.
	Thus, 
	\cite[Proposition~1.114]{Boris-I}, 
	\cref{l:F-subdif},
	and \cref{e:Ekeland-grad-g} imply that  
	\begin{equation}
	0 \in 
	\hat{\partial}\rb{g + \varepsilon  d_{ \left\{z\right\}}}\rb{z} 
	= 
	\nabla g\rb{z} + \varepsilon \ball{0}{1}
	= \frac{z-Tz}{\norm{z-Tz}} + \varepsilon \ball{0}{1}, 
	\end{equation}
	which is absurd since $\varepsilon \in \sqbo{0,1}$
	and $\norm{\rb{z-Tz}/ \rb {\norm{z-Tz}}} = 1 $.
\end{proof}

\begin{remark}
	Consider the setting 
	and the assumption of \cref{p:FixT-nonempty}.
	\begin{enumerate}
		\item Zarantonello
		established in 
		the proof of 
		\cite[Theorem~4.1]{zarontello1971projections-partI}
		that,  
		if (in addition to our assumption) 
		$T$ is Lipschitz continuous,
		then 
		$\Fix{T} \neq \varnothing$. 
		However, we do not
		need 
		the Lipschitz continuity of $T$
		in our proof. 
		\item 	Suppose, in addition,
		that $\nabla f$ is continuous.
		Then we obtain 
		an alternative proof
		as 
		follows.
            Assume
		to the contrary 
		that $\Fix T = \varnothing.$
		Then $g \coloneqq \sqrt{\, \cdot \,} \circ \rb{2f}$
		is continuously Fr\'{e}chet differentiable
		on $\HH$ (hence continuous)
		with 
		\begin{equation}\label{e:Ekeland-grad-g2}
		\rb{\forall x \in \HH}
		\quad 
		\nabla g\rb{x} 
		=	\frac{x - Tx}{\norm{x - Tx}}.
		\end{equation}
		Fix $\varepsilon \in \sqbo{0,1}$. 
		Since $g$ is bounded below
		and continuous,
		Ekeland's variational principle
		implies that
		there exists $z \in \HH$
		such that
            $\rb{\forall x \in \HH\smallsetminus\{z\}}\,
		g\rb{z} + \varepsilon d_{\left\{z\right\}}\rb{z} = g\rb{z} <   g\rb{x} + \varepsilon d_{\left\{z\right\}}\rb{x} $.
		Thus, $z$ is a
		minimizer of $g + \varepsilon d_{\left\{z\right\}}\rb{z}$.
		Therefore, 
		because $d_{\left\{z\right\}}$ is convex, 
		in view of \cite[Theorem~3.2.4(iii)\&(vi)\&(ii)]{Zalinescu-book2002}
		and \cite[Example~16.62]{bauschke2017convex},
		we see that 
		\begin{equation}
		0 \in \nabla g\rb{z} + \varepsilon \partial d_{\left\{z\right\}}\rb{z}
		=
		\nabla g\rb{z} + \varepsilon \ball{0}{1}
		= \frac{z-Tz}{\norm{z-Tz}} + \varepsilon \ball{0}{1}, 
		\end{equation}
		which contradicts
		the fact that $\varepsilon \in  \sqbo{0,1}. $
	\end{enumerate}

\end{remark}

By specializing
\cref{t:main}
to  positively 
homogeneous operators 
on $\HH$, 
we obtain
a characterization
for projectors
onto closed convex cones.

\begin{corollary}\label{c:Pr-cone}
	Let $T \colon \HH \to \HH$
	and set 
	$f \coloneqq \q \circ T$.
	Then the following are
	equivalent: 
	\begin{enumerate}
		\item\label{i:cone1} There exists 
		a nonempty closed convex cone $K$
		such that $T = P_{K}$.
		\item\label{i:cone2} $T$ is monotone and positively
		homogeneous, $f$ is
		G\^{a}teaux differentiable
		on $\HH$,
		and $\nabla f = T$.
	\end{enumerate}
	If \cref{i:cone1} or \cref{i:cone2} holds,
	then  
	$K =  \ran{T}$.
\end{corollary}

\begin{proof}
	``\cref{i:cone1}\ensuremath{\implies}\cref{i:cone2}'':
	Clearly $\ran{T} = \ran P_{K} = K$.
	Now,
	it follows from \cite[Example~20.32]{bauschke2017convex}
	that $T=P_{K}$ is monotone.
	Next, because $K$ is a nonempty closed convex cone,
	\cite[Proposition~29.29]{bauschke2017convex}
	guarantees that $T$ is positively
	homogeneous. In turn, since $f = \q \circ T = \q \circ P_{K}$, 
	\cite[Proposition~12.32
	and Lemma~2.61(i)]{bauschke2017convex}
	yield the G\^{a}teaux differentiability of $f$
	and, moreover, 
	$\nabla f = \nabla\rb{\q \circ P_{K}} = P_{K} =T$,
	as desired.

	``\cref{i:cone1}\ensuremath{\impliedby}\cref{i:cone2}'':
	First,
	since $T$ is
	positively
	homogeneous, 
	\begin{equation}\label{e:ranT-cone}
	\text{$\ran{T}$
		is a cone in $\HH$.}
	\end{equation}
	Now  
	set $g \colon \HH\to \RR : x \mapsto \rb{1/2}\scal{x}{Tx} 
	= \rb{1/2}\scal{x}{\nabla f\rb{x}}$
	and 
	\begin{math}
	h \coloneqq \q \circ \rb{\Id -T}.
	\end{math}
	Since $\nabla f = T$
	is monotone and positively homogeneous
	by assumption,
	\cref{l:derivative} ensures that
	$g$ is G\^{a}teaux differentiable
	on $\HH$ and 
	$\nabla g = \nabla f = T$.
	Thus, because 
	\begin{math}
	h = \q -2g + f,
	\end{math}
	it follows that $h$ is G\^{a}teaux
	differentiable on $\HH$ with gradient
	$\nabla h = \nabla \q -2\nabla g + \nabla f
	= \Id - 2T + T = \Id - T$.
	Consequently, since $T$ is monotone,
	we conclude via \cref{t:main}
	(applied to $h$)
	and \cref{e:ranT-cone} that
	$\ran{T}$ is a closed convex cone 
	in $\HH$ and that 
	$T = P_{\ran{T}}$.
\end{proof}

In \cref{c:Pr-cone},
if $T$ is a bounded linear
operator,
then we recover
the following 
characterization
of orthogonal projectors. 
For an alternative proof, 
which is based on
the orthogonal
decomposition $\HH = V \oplus V^{\perp}$,
where $V$ is a closed linear subspace
of $\HH$, 
see, e.g., 
\cite[Theorem~4.29]{weidmann2012}.

\begin{corollary}\label{c:pr-subspaces}
	Let $L \colon \HH\to \HH$.
	Then the following are equivalent:
	\begin{enumerate}
		\item\label{i:subspace1} There exists a
		closed linear subspace $V$ of $\HH$
		such that $L = P_{V}$.
		\item\label{i:subspace2} $L \in \mathscr{B}\rb{\HH}$ and 
		$L = L^{\ast} = L^{2}$.
		\item\label{i:subspace3}$L \in \mathscr{B}\rb{\HH}$
		and $L = L^{\ast}L$.
	\end{enumerate}
	If 
	one of 
	\cref{i:subspace1,i:subspace2,i:subspace3} holds,
	then $V = \ran{L}$.
\end{corollary}

\begin{proof}
	``\cref{i:subspace1}\ensuremath{\implies}\cref{i:subspace2}'': 
	See, e.g., 
	\cite[Corollary~3.24(iii)\&(vi)]{bauschke2017convex}.
	Moreover,
	it is clear that $\ran L = \ran P_{V} = V$.
	
	``\cref{i:subspace2}\ensuremath{\implies}\cref{i:subspace3}'':
	Clear.
	
	``\cref{i:subspace3}\ensuremath{\implies}\cref{i:subspace1}'':
	On the one hand, 
	because  $L \in \mathscr{B}\rb{\HH}$, 
	we deduce from 
	\cite[Example~20.16(ii)]{bauschke2017convex}
	that $L = L^{\ast}L$ is monotone. 
	On the other hand, since $L \in \mathscr{B}\rb{\HH}$, 
	\cite[Example~2.60]{bauschke2017convex}
	and our assumption 
	imply that $\q \circ L $ is Fr\'{e}chet 
	differentiable on $\HH$ and
	$\nabla\rb{\q \circ L} = L^{\ast}L = L$. 
	Altogether, 
	because  
	$L$ is clearly positively 
	homogeneous, 
	we obtain the conclusion
	via \cref{c:Pr-cone}.
\end{proof}

  \begin{theorem}[Linear combination of projectors]
    \label{t:conic-comb}
    Let $\fa{C_{i}}{i \in I} $
    be a finite family of
    nonempty closed convex subsets
  of $\HH$,
    let $\fa{\alpha_{i}}{i \in I} $
    be a family in $\RR$, and
    set $\alpha \coloneqq \sum_{ i \in I}\alpha_{i}$.
    Then,
    there exists a closed
    convex set $C$ such that 
    $\sum_{ i \in I}\alpha_{i}P_{C_{i}} = P_{C}$
    if and only if
$\sum_{ i \in I}\alpha_{i}P_{C_{i}}$ is monotone
and 
\begin{equation}
  \label{eq:cond-conical}
  \rbr{\exists \gamma \in \RR} \rbr{\forall x \in \HH } \quad
\rbr{\alpha-1} \sum_{i \in I}\alpha_{i}\q\rbr{P_{C_{i}}x} 
      - \tfrac{1}{2}\sum_{i \in I}\sum_{j \in I}\alpha_{i}\alpha_{j}
      \q\rbr{P_{C_{i}}x - P_{C_{j}}x} = \gamma; 
\end{equation}
in which case,
\begin{equation}
  \label{eq:dist-general}
  d_{C}^{2} = \sum_{i \in I}\alpha_{i}d_{ C_{i} }^{2}
  - 2\rb{\alpha-1}\q +2\gamma. 
\end{equation}
  \end{theorem}

  \begin{proof}
    Set $T \coloneqq \sum_{i \in I}\alpha_{i}P_{C_{i}}$,
    set $f \coloneqq \q \circ \rbr{\Id - T}$,
    and define
    \begin{equation}
      \label{eq:defn-g-linear}
      g \colon \HH \to \RR : 
      x\mapsto 
\rbr{\alpha-1} \sum_{i \in I}\alpha_{i}\q\rbr{P_{C_{i}}x} 
      - \tfrac{1}{2}\sum_{i \in I}\sum_{j \in I}\alpha_{i}\alpha_{j}
      \q\rbr{P_{C_{i}}x - P_{C_{j}}x}.
    \end{equation}
    In view of \cref{p:sumCi}\cref{i:sum-many1},
    we have 
    \begin{equation}
      \label{eq:expand-f}
      \rbr{\forall x \in \HH}\quad 
      f\rbr{x} = \tfrac{1}{2}\sum_{ i \in I}\alpha_{i}d_{C_{i}}^{2}\rbr{x}
      - \rbr{\alpha-1}\q\rbr{x} + g\rbr{x}.
    \end{equation}
Now assume that
there exists a 
nonempty closed convex subset $C$
of $\HH$ such that
$T = P_{C}$.
Then,  due to \cref{eg:PC-mono}\cref{i:Pc-maxmono},
we see that $T$ is monotone.
Next, on the one hand, 
since $T = P_{C}$,
it follows from \cref{t:main} that 
$f$ is Fr\'echet differentiable on $\HH$
and $\nabla f = \Id - T =\Id - \sum_{i \in I}\alpha_{i}P_{C_{i}}$.
On the other hand, for every $i\in I$,
since $C_{i}$ is convex, closed, and nonempty,
we infer
from \cref{t:main}
(applied to $P_{C_{i}}$)
that $d_{C_{i}}^{2} =
2\q \circ \rbr{\Id - P_{C_{i}}}$ is Fr\'echet
differentiable on $\HH$ with
$\nabla d_{C_{i}}^{2} = 2\rbr{\Id - P_{C_{i}}}$.
Altogether,
since $\alpha = \sum_{i \in I }\alpha_{i}$
by definition,
it follows from \cref{eq:expand-f}
that $g$ is Fr\'echet differentiable on $\HH$
and that 
\begin{equation}
  \nabla g = \nabla f - \nabla \rbr[\Big]{
  \tfrac{1}{2}\sum_{i \in I}\alpha_{i}d_{C_{i}}^{2} - \rbr{\alpha-1}\q}
  = \rbr[\Big]{\Id - \sum_{i \in I}\alpha_{i}P_{C_{i}} } - \sum_{i\in I}
  \alpha_{i}\rbr{\Id - P_{C_{i}}} + \rbr{\alpha-1}\Id  
  =0.
\end{equation}
Consequently, there exists $\gamma \in \RR$
such that $\rbr{\forall x \in \HH}\, g\rbr{x} = \gamma$.
Conversely, assume that $T$
is monotone and that \cref{eq:cond-conical} holds.
Then, we derive from \cref{eq:expand-f} that 
\begin{equation}
  \label{eq:expand-f2}
  f = \tfrac{1}{2}\sum_{i \in I}\alpha_{i}d_{C_i}^{2} 
  - \rbr{\alpha-1}\q
+\gamma,
\end{equation}
and it  thus follows that  
$f$ is Fr\'echet differentiable on $\HH$
and, since $\alpha = \sum_{ i \in I}\alpha_{i}$,
 $\nabla f = \sum_{i \in I}\alpha_{i}\rbr{\Id - P_{C_{i}} }
- \rbr{\alpha-1}\Id = \Id - 
\sum_{i \in I}\alpha_{i}P_{C_{i}}
 =\Id -  T
 $.
 Hence, since $T$ is monotone by
 our assumption, 
 \cref{t:main} ensures the
 existence of a nonempty closed convex set
 $C$ such that $T = P_{C}$.
 Therefore,
 $f= \q \circ \rbr{ \Id - P_{C}} = \rbr{1/2}d_{C}^{2}$
 and \cref{eq:dist-general} follows from
 \cref{eq:expand-f2}.
  \end{proof}

  \begin{remark}
    \label{rm:linear-set}
    As we have seen in 
    \cref{rem:counter-sum},
    the set $C$ in \cref{t:conic-comb}
    need not be $\sum_{i \in I}\alpha_{i}C_{i}$.
  \end{remark}

We now establish
a necessary and sufficient
condition under which
a finite sum of projectors
is a projector.

\begin{theorem}[Sum of projectors]\label{t:sum-many}
  Let
	$\fa{C_{i}}{ i \in I }$
	be a finite family of nonempty 
      closed convex subsets of $\HH$, and 
  set $\alpha \coloneqq \card I$.
	Then 
	$\sum_{ i \in I} P_{C_{i}}  \in \prH$
	if and only if
	\begin{equation}
        \label{eq:cond-sum}
	\rb{\exists \gamma \in \RR}\rb{\forall x \in \HH}
	\quad
	\sum_{\substack{\rb{i,j} \in I\times I  \\ i \neq  j}}
	\scal{ P_{C_{i}}x }{ P_{C_{j}}x} = \gamma;
	\end{equation}
	in which case, $\sum_{ i \in I} C_{i}$ is 
	a closed convex set,
		\begin{equation}\label{e:sum-pr}
			\sum_{i \in I} P_{C_{i}} = P_{\sum_{i \in I} C_{i} },
		\end{equation}
	and 
	\begin{equation}\label{e:dist-sum-pr}
	d_{\sum_{ i \in I} C_{i}}^{2} = \sum_{i \in I}d_{C_{i}}^{2}
      - 2\rb{\alpha-1}\q + \gamma.
	\end{equation}
\end{theorem}

\begin{proof}
  Since it
  is clear that $\sum_{ i \in I}P_{C_{i}}$
  is monotone, 
  we derive from \cref{t:conic-comb}
  (applied to $\fa{C_{i}}{i \in I}$,
$\fa{\alpha_{i}}{i \in I} = \fa{1}{i \in I}$,
and $\alpha = \card I = \sum_{i \in I}1$)
  and \cref{eq:iden-21} that 
\begin{subequations}
  \begin{align}
    \sum_{ i \in I}P_{C_{i}} \in \prH 
    & \iff \rbr{\exists \gamma \in \RR}\rbr{\forall x \in \HH}
    \; \rbr{\alpha-1}\sum_{i \in I}\q\rbr{P_{C_{i}}x}
    -\tfrac{1}{2}\sum_{i \in I}\sum_{j \in I}\q\rbr{P_{C_{i}}x - P_{C_{j}}x}
    =\gamma \\
    & \iff \rbr{\exists \gamma \in \RR}\rbr{\forall x \in \HH}\;
    \tfrac{1}{2}\sum_{\substack{\rbr{i,j}\in I\times I \\ i \neq j}}
    \scal{P_{C_{i}}x}{P_{C_{j}}x} =\gamma 
    \\
    & \iff \cref{eq:cond-sum},
  \end{align} 
\end{subequations}
  as desired.
  Next, suppose that 
  $\sum_{i \in I}P_{C_{i}} \in \prH$.
  Then, there exists a closed convex set $C$
  such that 
  \begin{equation}
    \label{eq:sum-in}
    \sum_{i \in I}P_{C_{i}} = P_{C};
  \end{equation}
  therefore,
  as we have shown above,
  there exists $\gamma \in \RR$
  such that 
  \begin{equation}
    \label{eq:constant}
\rbr{\forall x \in \HH}\quad  
\sum_{\substack{ \rbr{i,j} \in I\times I \\
i\neq j}}\scal{P_{C_{i}}x}{P_{C_{j}}x} = \gamma.
  \end{equation}
  According to \cref{p:sumCi}\cref{i:set-sum2}
  and \cref{eq:sum-in},
  we see that $ \sum_{i \in I}C_{i} = C$ is a
  closed convex set, from which
  and \cref{eq:sum-in} we get 
  \cref{e:sum-pr}. Furthermore,
  it follows from \cref{eq:iden-22}
  and  \cref{eq:constant} that 
  \begin{subequations}
    \begin{align}
    \rbr{\forall x \in \HH}\quad 
    d_{C}^{2}\rbr{x} = \norm{x - P_{C}x}^{2}
    & = \norm[\Big]{x- \sum_{i \in I}P_{C_{i}}x}^{2}
    \\
    & = \rbr{1-\alpha}\norm{x}^{2} + \sum_{i \in I}\norm{x-P_{C_{i}}x}^{2}
    + \sum_{\substack{\rbr{i,j}\in I\times I \\ i \neq j}}
    \scal{P_{C_{i}}x}{P_{C_{j}}x}
    \\
    & = \sum_{i \in I}d_{C_{i}}^{2}\rbr{x}
    - 2\rbr{\alpha-1}\q\rbr{x} +\gamma,
  \end{align}
  \end{subequations}
  and \cref{e:dist-sum-pr} follows.
\end{proof}

\begin{corollary}\label{t:sum-pr}
	Let $C$ and $D$ 
	be nonempty closed convex subsets 
	of $\HH$.
	Then the following are equivalent:
	\begin{enumerate}
		\item\label{i:sum1}
		$P_{C} + P_{D} \in \prH$.
		\item\label{i:sum2}
              $\rb{\exists \gamma \in \RR}\rb{\forall x \in \HH}\,
		\scal{P_{C}x }{ P_{D}x } = \gamma$.
	\end{enumerate}
	If \cref{i:sum1} or \cref{i:sum2} holds, 
	then $C+D$ is a
	closed convex set,  
	\begin{equation}
	P_{C} + P_{D} = P_{C+D}, 
	\end{equation}
	and 
	\begin{equation}\label{e:dist-sum}
	d_{C+D}^{2} = d_{C}^{2} + d_{D}^{2} - 2\q + 2\gamma.
	\end{equation}
\end{corollary}

  \begin{remark}
    Consider the setting of \cref{t:sum-pr}.
    In view of \cite[Example~12.3]{bauschke2017convex},
    we see that \cref{e:dist-sum} is equivalent
    to 
    $\rbr{\iota_{C}\infconv \iota_{D}}\infconv \q 
    = \iota_{C}\infconv \q + \iota_{D}\infconv \q -\q +\gamma$.
    Hence, using \cite[Example~13.3(i) and 
    Proposition~13.24(i)]{bauschke2017convex}
    and Moreau's decomposition \cite{Moreau-prox-dual-1965},
    we infer that 
    \begin{subequations}
      \begin{align}
        \cref{e:dist-sum} 
        & \iff \q - \rbr{\iota_{C}\infconv \iota_{D}}^{\ast}\infconv\q
        = -  \rbr{\iota_{C}^{\ast}\infconv \q } 
        - \rbr{\iota_{D}^{\ast}\infconv \q} + \q 
        +\gamma
        \\
        & \iff 
        \q - \rbr{\iota_{C}^{\ast} + \iota_{D}^{\ast}}\infconv \q
        = 
         -  \rbr{\iota_{C}^{\ast}\infconv \q } 
        - \rbr{\iota_{D}^{\ast}\infconv \q}
        +\q 
        +\gamma
        \\
        & \iff 
\rbr{\iota_{C}^{\ast} + \iota_{D}^{\ast}}\infconv \q
= \iota_{C}^{\ast}\infconv\q + \iota_{D}^{\ast}\infconv\q - \gamma. 
\label{eq:prox-equi}
      \end{align} 
    \end{subequations}
This type of relationship
is used in \cite[Proposition~3.16]{Combettes-2018-mono} 
to establish a condition
for the sum of two 
proximity operators 
to be a proximity operator.
  \end{remark}

The following 
simple example shows 
that the constant $\gamma$
in \cref{t:sum-pr} can take 
on any value.

\begin{example}
	Let $u$ and $v$ be in $\HH$,
	set $C \coloneqq \{ u \}$,
	and set $D \coloneqq \{v\}$.
	Then clearly $P_{C} + P_{D} = P_{\{u+v\}} = P_{C+D}$
      and $\rb{\forall x \in \HH}\, 
	\scal{P_{C}x}{P_{D}x} = \scal{u}{v}$. 
\end{example}

As a consequence of \cref{t:sum-pr},
a sum of projectors
onto orthogonal sets
is a projector; 	
see 
\cite[Proposition~2.6]{bauschke2006strongly}
for a difference
derivation.

\begin{corollary}\label{c:CperpD}
	Let $C$
	and $D$ be nonempty closed
	convex subsets of $\HH$
	such that $C \perp D$.
	Then the following hold: 
	\begin{enumerate}
		\item $C+D$ is a nonempty closed
		convex set.
		\item $P_{C} + P_{D} = P_{C+D}$.
		\item $d_{C+D}^{2} = d_{C}^{2} + d_{D}^{2} - 2\q$.
	\end{enumerate}
\end{corollary}

\begin{proof}
  Since $\rb{\forall x \in \HH}\, \scal{P_{C}x }{ P_{D}x  } = 0$, the 
	conclusions
	readily 
	follow  from \cref{t:sum-pr}.
\end{proof}

We now provide an 
instance  
where item 
\cref{i:sum2} of \cref{t:sum-pr}
holds, $C \nsubseteq D^{\perp}$
in general,
and neither $C$ nor $D$
is a cone.

\begin{example}\label{eg:figgy}
	Let $K$ be a 
	nonempty closed
	convex cone 
	in $\HH$,
	let $\rho_{1}$
	and $\rho_{2}$
	be in $\RPP$,
	set $C \coloneqq K \cap \ball{0}{\rho_{1}}$,
	and set $D \coloneqq \pc{K} \cap 
	\ball{0}{\rho_{2}}$. 
	It then immediately follows
	from 
	\cite[Theorem~7.1]{bauschke2017projecting}
	and \cite[Theorem~6.30(ii)]{bauschke2017convex}
	that 
		\begin{equation}
			\rb{\forall x \in \HH}
			\quad 
			\scal{P_{C}x }{ P_{D}x } 
			= \scal*{ \frac{\rho_{1}}{ \max\left\{ \norm{P_{K}x},\rho_{1} \right\} } P_{K}x }{
			\frac{\rho_{2}}{ \max\left\{ \norm{P_{ \pc{K} }x},\rho_{2} \right\} } P_{ \pc{K}}x }
			= 0 .
		\end{equation}
\end{example}

\begin{figure}[H]
	\centering 
	\includegraphics[scale= 0.45]{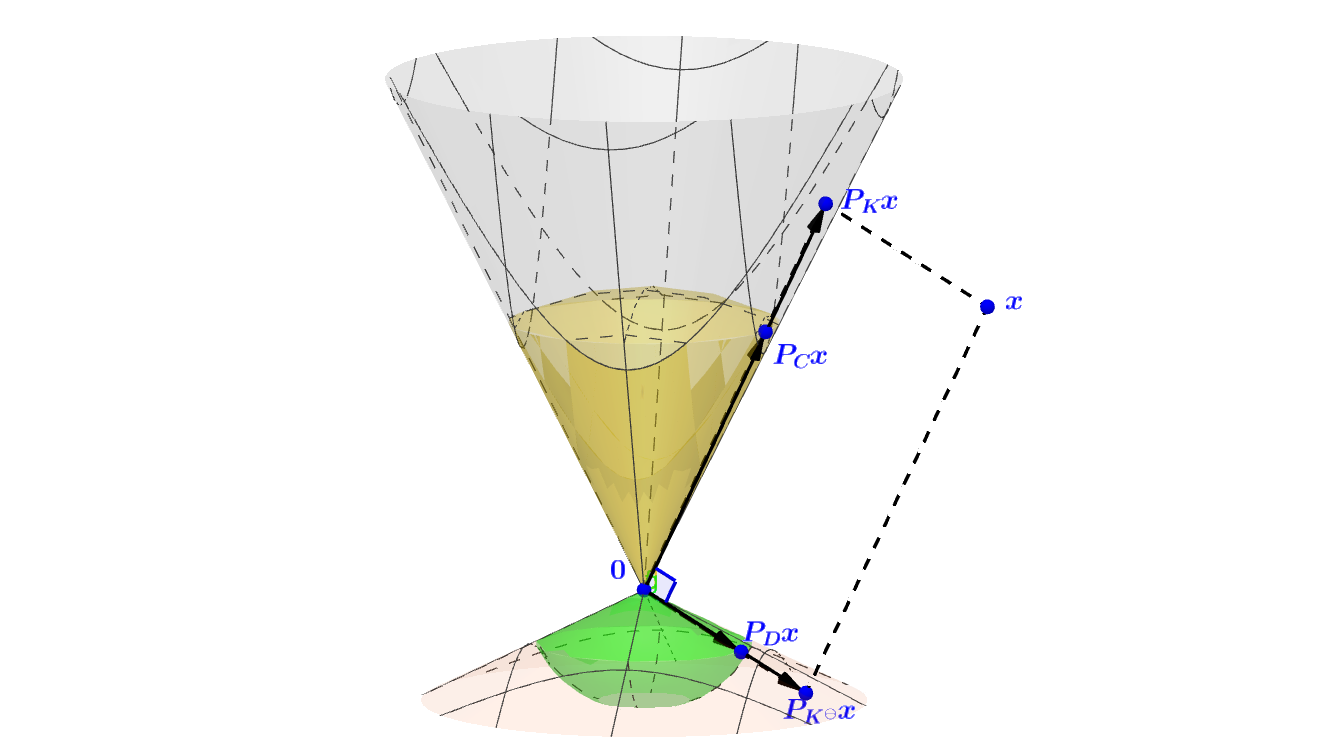}
	\caption{
	A \texttt{GeoGebra} 
	\cite{geogebra}
	snapshot
	illustrating
	the sets $C$ (yellow)
	and $D$ (green) in the setting
	of \cref{eg:figgy}.}
\end{figure}

We next establish a necessary and sufficient condition
for  $u+P_{C}$ to be a projector.

\begin{example}\label{eg:uPC}
	Let $C$ be 
	a nonempty closed
	convex subset of $\HH$,
	and let $u \in \HH.$
      Then, since $\rb{\forall x \in \HH}\,u = P_{ \left\{u\right\} }x$, 
	we deduce from \cref{t:sum-pr} that 
	\begin{subequations} 
		\begin{align}
		u + P_{C} = P_{\left\{u\right\}} + P_{C} \in \prH  
		& \iff \rb{\exists \gamma \in \RR}\rb{\forall x \in \HH}
            \, \scal{u}{P_{C}x} = \gamma\\
            & \iff \rb{\exists \gamma \in \RR}\rb{\forall x \in C}\,
		\scal{u}{x} = \gamma \\
		& \iff \rb{\forall x \in C}\rb{\forall y \in C}
            \, \scal{u}{x} = \scal{u}{y} \\
		&  \iff \rb{\forall x \in C}\rb{\forall y \in C}
            \, \scal{u}{x-y} = 0 \\
		& \iff u \in \rb{C-C}^{\perp};
		\end{align}
	\end{subequations}
in which case, $u + P_{C} = P_{u+ C}$
due to \cref{t:sum-pr}.
\end{example}

\begin{remark}
	Consider the setting of \cref{eg:uPC}.
	Since $u + P_{C}$
	is monotone,  nonexpansive,
	and a sum of proximity operators, 
	\cite[Corollary~2.5]{bartz2017resolvent}
	guarantees that $u + P_{C}$
	is 
	a proximity operator.
	However, 
	by \cref{eg:uPC}, 
	it is not a projector
	unless $u \in \rb{C-C}^{\perp}.$ 
\end{remark}

Here is
a sufficient,
but not necessary, 
condition
for a sum of projectors to be
a projector.

\begin{corollary}\label{c:induct-many}
	Let $m \geq 2$
	be an integer,
	set $I \coloneqq \{1,\ldots,m\}$,
	let 
	$\fa{C_{i}}{ i \in I }$
	be a family of nonempty 
	closed convex subsets of $\HH$,
	and set $C \coloneqq \sum_{ i \in I} C_{i}.$
	Suppose that, 
	for every $\rb{i,j} \in I \times I$
	with $i<j$,
	there exists $\gamma_{i,j} \in \RR$
	such that 
      $\rb{\forall x \in \HH}\,\scal{P_{C_{i}}x }{ P_{C_{j}}x }= \gamma_{i,j}$.
	Then $C$ is a closed convex set and 
	\begin{math}
	\sum_{i \in I} P_{C_{i}} = P_{C}. 
	\end{math}
\end{corollary}

	\begin{proof}
        Set $\rb{\forall k \in I}\, D_{k} \coloneqq \sum_{i=1}^{k} C_{i}$,
		and let us establish that 
		\begin{equation}\label{e:inductive-many}
		\rb{\forall k \in I\smallsetminus \{1\}} \quad 
		\quad  
		\text{$D_{k}$ is a closed convex set and}
		\sum_{i=1}^{k} P_{C_{i}} = P_{D_{k}}.
		\end{equation}
		Due to \cref{t:sum-pr}, the claim holds if $k=2$,
		and we therefore assume that, for some $k \in \{2,\ldots,m-1\}$,
		$D_{k}$ is a closed convex set and that 
		$\sum_{i=1}^{k} P_{C_{i}} = P_{D_{k}}$. 
		Then, by our assumption, 
            $\rb{\forall x \in \HH}\,\scal{ P_{D_{k}}x }{ P_{C_{k+1}}x } = \sum_{i=1}^{k}\scal{ P_{C_{i}}x }{ P_{C_{k+1}}x  } = \sum_{i=1}^{k}\gamma_{i,k+1}$,
		from which and \cref{t:sum-pr} (applied to $D_{k}$ and $C_{k+1}$)
		we infer that $D_{k+1} = D_{k} + C_{k+1}$ is a closed
		convex set and, due to
		the induction hypothesis, 
		$\sum_{i=1}^{k+1}P_{C_{i}} 
		= \sum_{i=1}^{k} P_{C_{i}} + P_{C_{k+1}} 
		= P_{D_{k}} + P_{C_{k+1}}
		= P_{D_{k} + C_{k+1} } 
		= P_{D_{k+1}} $.
		Hence, 
		 letting $k=m$ in \cref{e:inductive-many} 
		yields the conclusion.
	\end{proof}

We now illustrate that 
the assumption of \cref{c:induct-many}
need not hold 
when merely $\sum_{ i \in I} P_{C_{i}} = P_{C}$.

\begin{example}
	Let $C$ be a nonempty closed convex
	subset of $\HH$ 
	such that $\HH \smallsetminus \rb{C - C}^{\perp }
	\neq \varnothing$, 
	and suppose that 
	$u \in \HH \smallsetminus \rb{C - C}^{\perp }
	.$
	Then $P_{ \left\{u\right\}} + P_{\left\{-u\right\}} + P_{C} 
	= P_{C}$
	is a projector.
	However, 
	if 
	$x\mapsto \scal{P_{ \{u\} }x}{ P_{C}x} = \scal{u}{P_{C}x}$
	were a constant,
	then it would follow 
	from \cref{t:sum-pr}
	that $u + P_{C} = P_{\left\{u\right\}} + P_{C}$
	is a projector, 
	which 
	violates
	\cref{eg:uPC}
	and the assumption 
	that $u \notin \rb{C-C}^{\perp}.$
\end{example}

We conclude this section
with a result concerning
the difference of two projectors.

\begin{proposition}
      \label{p:difference-prs}
      Let $C$
      and $D$ be nonempty closed convex subsets
      of $\HH$.
      Then $P_{D} - P_{C} \in \prH$
      if and only if 
      $P_{D} - P_{C}$ is monotone
      and there exists
      $\gamma \in \RR$ such that 
      $\rbr{
      \forall x \in \HH} \, \scal{P_{C}x}{P_{D}x - P_{C}x} = \gamma$.
\end{proposition}

\begin{proof}
      Using \cref{t:conic-comb}
      with $I=\left\{ 1,2 \right\}$,
      $\rbr{C_{1},C_{2}} = \rbr{D,C}$,
      and $\rbr{\alpha_{1},\alpha_{2}} = \rbr{1,-1}$,
      we infer that 
      $P_{D} - P_{C} \in \prH$
      if and only if $P_{D} - P_{C}$
      is monotone and there exists $\gamma \in \RR$
      such that 
      $\rbr{\forall x \in \HH} \, 
      {-}\gamma =
      -\rbr{\q\rbr{ P_{D}x} - \q\rbr{P_{C}x}}
      + \q\rbr{P_{C}x - P_{D}x}
      = \scal{P_{C}x}{P_{C}x - P_{D}x}$,
      which is the desired conclusion.
\end{proof}

\section{Convex combination of projectors}
\label{sect:convex-comb}

The analysis of this section
requires the following
results.

\begin{fact}[Zarantonello]
  \label{f:Zaran-convex-comb}
  Let $\fa{T_{i}}{i \in I} $ be a finite
  family of firmly nonexpansive operators
  from $\HH$ to $\HH$,
  let $\fa{\alpha_{i}}{i \in I} $
  be real numbers in $\left]0,1\right]$
  such that $\sum_{i \in I}\alpha_{i} =1$,
  and let $C$
  be a nonempty closed convex subset 
  of $\HH$.
  Then $\sum_{i \in I}\alpha_{i}T_{i} = P_{C}$ 
  if and only if
  there exist vectors $\fa{u_{i}}{i \in I} $
  in $\HH$ such that 
  $\rbr{\forall i \in I} \, T_{i} = P_{C} + 
  u_{i}$
  and  $\sum_{i \in I}\alpha_{i}u_{i} = 0$.
\end{fact}

\begin{proof}
  See \cite[Theorem~1.3]{zarontello1971projections-partI}.
\end{proof}

\begin{lemma}
  \label{f:BCL-2004}
  Let $C$ and $D$ be nonempty 
  closed convex subsets of $\HH$, and
  set $v \coloneqq P_{\closu{D-C}}0$.
  Then the following hold:
  \begin{enumerate}
    \item\label{i:BCL-2004}  
   Let $\fa{c_{n}}{n \in \NN}$
  and $\fa{d_{n}}{n \in \NN}$ be
  sequences in $C$ and $D$, 
  respectively,
  and
  suppose that $d_{n}-c_{n}\to v$.
  Then 
  $d_{n} - P_{C}d_{n} \to v$.
\item\label{i:2sets-gap} Suppose that 
  there exists $u \in \HH$
  such that $P_{D} = P_{C} + u$.
  Then $u=v \in \rbr{C-C}^{\perp}$ and $D = C + v$.
  \item\label{i:2sets-distant}
        Suppose that 
        there exists $\gamma \in \RR$
        such that 
        $\rbr{\forall x \in \HH} \, 
        \norm{P_{D}x - P_{C}x} = \gamma$.
        Then $v \in \rbr{C - C}^{\perp}$
        and $D = C + v$.
  \end{enumerate}
\end{lemma}

\begin{proof}
  \cref{i:BCL-2004}: See \cite[Proposition~2.5(i)]{bauschke2004finding}.

  \cref{i:2sets-gap}: 
  Since $P_{D} = P_{C} + u$,
  \cref{eg:uPC} guarantees that 
     $u \in \rbr{C-C}^{\perp}$ and that 
     $D = C + u$.
  Hence, it suffices to show that $u=v$.
  Indeed,
  since $v= P_{\closu{D-C}}0 \in \closu{D-C}$,
  there exist sequences $\fa{c_{n}}{n \in \NN}$
  in $C$ and $\fa{d_{n}}{n \in \NN}$ in $D$
  such that $d_{n} - c_{n} \to v$.
  Thus, we deduce from \cref{i:BCL-2004} that
  \begin{equation}
    \label{eq:limit-gap}
    d_{n} - P_{C}d_{n} \to v.
  \end{equation}
  On the other hand, since $P_{D} = P_{C} + u$
  and $\fa{d_{n}}{n \in \NN}$ is a sequence in $D$,
  it follows that $\rbr{\forall n \in \NN}\, u = P_{D}d_{n} - P_{C}d_{n} = d_{n} - P_{C}d_{n}$.
  This and \cref{eq:limit-gap}
  yield $u=v$, as claimed.

  \cref{i:2sets-distant}: 
  Let $\fa{c_{n}}{n \in \NN}$ and 
  $\fa{d_{n}}{n \in \NN}$ be
  sequences in $C$ and $D$, respectively,
  such that $d_{n} - c_{n} \to v$.
  According to \cref{i:BCL-2004},
  $d_{n} - P_{C}d_{n} \to v$,
  and therefore,
  $\norm{d_{n} - P_{C}d_{n}} \to \norm{v}$.
  However, since $\rbr{\forall n \in \NN} \, 
  d_{n} \in D$,
  it follows from our assumption that
  $\rbr{\forall n \in \NN}\, 
  \gamma = \norm{P_{D}d_{n} - P_{C}d_{n}} 
  = \norm{d_{n} - P_{C}d_{n}}$.
  Hence, 
  invoking the assumption once more,
  $\rbr{\forall x \in \HH} \, 
  \norm{P_{\closu{D-C}}0} = 
  \norm{v} = \gamma = \norm{P_{D}x - P_{C}x}$.
  Consequently, since $\rbr{\forall x \in \HH} \,
  P_{D}x - P_{C}x \in \closu{D-C}$, 
  we conclude via 
  \cite[Lemma~2.4]{bauschke-moursi-2016} that $\rbr{\forall x \in \HH} \,
  P_{D}x - P_{C}x = P_{\closu{D-C}}0 = v$.
  Now apply \cref{i:2sets-gap}.
\end{proof}

Here is our main result
of this section.

\begin{theorem}[Convex combination
  of projectors]
  \label{t:convex-comb}
  Let $\fa{C_{i}}{i \in I}$ be a finite 
  family of nonempty closed convex
  subsets of $\HH$,
  let $k \in I$, and 
set $\rbr{\forall i \in I} \, v_{i} \coloneqq 
P_{\closu{C_{i}-C_{k}}}0$.
Then the following are equivalent: 
\begin{enumerate}
  \item\label{i:convex-1} 
        There exists $\fa{\alpha_{i}}{i \in I}$
  in $\left]0,1 \right]^{I}$
  such that $\sum_{i \in I}\alpha_{i}=1$
  and $\sum_{i \in I}\alpha_{i}P_{C_{i}} \in 
  \prH$.
  \item\label{i:convex-2} For every $\rbr{i,j} \in I \times I$,
    there exists $\alpha \in \RR \smallsetminus \left\{ 0,1 \right\}$
    such that $\rbr{1-\alpha}P_{C_{i}} + \alpha P_{C_{j}} \in 
    \prH$.
  \item\label{i:convex-3} 
        For every $i \in I$, 
        we have
        $v_{i} \in \rbr{C_{k}-C_{k}}^{\perp}$
        and $C_{i} = C_{k}+ v_{i}$.
  \item\label{i:convex-4} $\menge[\Big]{\sum_{i \in I}\alpha_{i}P_{C_{i}}}{
    \fa{\alpha_{i}}{i \in I} \in \RR^{I}\text{~and~} 
    \sum_{i \in I}\alpha_{i} = 1
  }\subset \prH$.
\end{enumerate}
Furthermore, each of the above implies that,
for every $\fa{\alpha_{i}}{i \in I} \in \RR^{I}$
such that $\sum_{i\in I}\alpha_{i}=1$, we have
\begin{equation}
      \label{eq:affine-comb}
      \sum_{ i\in I}\alpha_{i}P_{C_{i}} = P_{C_{k}
      + \sum_{i \in I}\alpha_{i}v_{i} }.
\end{equation}
\end{theorem}

\begin{proof}
  ``\cref{i:convex-1}\ensuremath{\implies}\cref{i:convex-3}'': 
  Suppose that there exist 
  $\fa{\alpha_{i}}{i \in I} \in \left] 0,1 \right]^{I}$
  and a nonempty closed convex subset 
  $C$
  of $\HH$ such that 
  $\sum_{i \in I}\alpha_{i} = 1$ and 
  $\sum_{i \in I} \alpha_{i}P_{C_{i}} = P_{C}$.
  Then,
  since $\fa{P_{C_{i}}}{i \in I}$
  are firmly nonexpansive by \cite[Proposition~4.16]{bauschke2017convex},
\cref{f:Zaran-convex-comb} guarantees the
existence of vectors $\fa{u_{i}}{i \in I}$ in $\HH$
such that 
\begin{equation}
  \label{eq:Ci-C}
  \rbr{\forall i \in I} \quad  P_{C_{i}} = P_{C } + u_{i}.
\end{equation}
Now fix $ i \in I$. We then derive from  \cref{eq:Ci-C}
that $P_{C_{i}} = \rbr{P_{C_{k}} - u_{k}} + u_{i}
= P_{C_{k}} + u_{i} - u_{k}$,
and it thus follows from
 \cref{f:BCL-2004}\cref{i:2sets-gap}
 (applied to $\rbr{C_{k},C_{i},u_{i}-u_{k}}$) that 
  $v_{i} \in \rbr{C_{k}-C_{k}}^{\perp}$ and  
  $C_{i} = C_{k}+v_{i}$, as required.

``\cref{i:convex-3}\ensuremath{\implies}\cref{i:convex-4}'':
Let $\fa{\alpha_{i}}{i \in I} \in \RR^{I}$ be
such that $\sum_{i \in I}\alpha_{i} =1$. 
Then,
since $\rbr{\forall i \in I} \, v_{i} \in 
\rbr{C_{k} - C_{k}}^{\perp}$,
it follows that $\sum_{i \in I}\alpha_{i}v_{i} \in 
\rbr{C_{k} - C_{k}}^{\perp}$.
In turn,
because $\sum_{i \in I}\alpha_{i}=1$,
our assumption and 
\cref{eg:uPC} yield 
$\sum_{ i \in I}\alpha_{i}P_{C_{i}} = \sum_{i \in I}
\alpha_{i}P_{C_{k}+v_{i}}
= \sum_{i \in I}\alpha_{i}\rbr{P_{C_{k}} + v_{i}}
=P_{C_{k}} + \sum_{i \in I}\alpha_{i}v_{i} 
= P_{C_{k} + \sum_{i \in I}\alpha_{i}v_{i}}$, 
which establishes \cref{i:convex-4} and \cref{eq:affine-comb}.

``\cref{i:convex-4}\ensuremath{\implies}\cref{i:convex-1}'':
Clear.

At this point, we have shown that 
\begin{equation}
      \label{eq:equivalences}
      \cref{i:convex-1}\iff \cref{i:convex-3}\iff 
      \cref{i:convex-4}\implies\cref{eq:affine-comb}.
\end{equation}
To complete the proof,
we shall show that
\cref{i:convex-2}$\iff$\cref{i:convex-3}.

``\cref{i:convex-2}\ensuremath{\implies}\cref{i:convex-3}'':
Fix $i \in I$. Then, by assumption,
there exists $\alpha \in \RR\setminus\left\{ 0,1 \right\}$
such that 
$\rbr{1-\alpha}P_{C_{i}} + \alpha P_{C_{k}} \in 
\prH$.
Therefore,
applying \cref{t:conic-comb}
to $\rbr{C_{i},C_{k}}$
and the corresponding coefficients
$\rbr{1-\alpha,\alpha}$,
we deduce the existence of 
$\gamma \in \RR$ such that 
$\rbr{\forall x \in \HH} \, 
\rbr{1-\alpha}\alpha \norm{P_{C_{i}}x - P_{C_{k}}x}^{2} = \gamma$.
Thus, because $\rbr{1-\alpha}\alpha \neq 0$
due to the fact that $\alpha \in \RR \setminus \left\{ 0,1 \right\}$,
it follows that $\rbr{\forall x \in \HH} \, 
\norm{P_{C_{i}}x - P_{C_{k}}x}^{2} = 
\gamma \alpha^{-1}\rbr{1-\alpha}^{-1}$.
This and 
\cref{f:BCL-2004}\cref{i:2sets-distant} yield 
\cref{i:convex-3}.

``\cref{i:convex-3}\ensuremath{\implies}\cref{i:convex-2}'':
Suppose that \cref{i:convex-3} holds. Then,
due to \cref{eq:equivalences}, \cref{i:convex-4} holds,
from which \cref{i:convex-2} follows.
\end{proof}

The following example shows that
the conclusion of 
\cref{t:convex-comb} fails
if we
replace 
``convex combination'' by 
``affine combination''
in item \cref{i:convex-1}. 
\begin{example}
      Let $C$
      be a nonempty closed convex
      subset of $\HH$,
      and let $u \in \HH$. 
      Then 
      the affine combination
      of $\rbr{P_{C},P_{C},P_{ \left\{ u \right\}}}$
      with weights $\rbr{1/4,-1/4,1}$
      is a projector
      since
      $
      \rbr{1/4}P_{C} - \rbr{1/4}P_{C} + P_{ \left\{ u \right\} } 
      = P_{\left\{ u \right\} }$.
However, \cref{t:convex-comb}\cref{i:convex-3} fails
when $C$ is not a singleton.
\end{example}

Here are some direct consequences
of \cref{t:convex-comb}. 
\begin{corollary}
  \label{c:nonempty-intersection}
  Let $\fa{C_{i}}{i \in I}$ 
  be a finite family of nonempty closed
  convex subsets of $\HH$. Suppose that 
 $\cap_{i \in I}C_{i} \neq \varnothing$
 and that 
 there exists $\fa{\alpha_{i}}{i \in I} \in \left] 0,1\right]^{I}$
 such that $\sum_{i \in I}\alpha_{i} = 1$ and
 $\sum_{i \in I}\alpha_{i}P_{C_{i}} \in \prH$.
 Then $\rbr{\forall i \in I}\rbr{\forall j \in I} \, C_{i} = C_{j}$.
\end{corollary}

\begin{proof}
      Let $k \in I$ and let $i \in I$.
      Since $C_{k} \cap C_{i} \neq \varnothing$
      by assumption, we see that
      $P_{\closu{C_{i}  - C_{k}}}0 = 0 $,
      and thus, 
      due to our assumption,
      the implication ``\cref{i:convex-1}\ensuremath{\implies}\cref{i:convex-3}''
      of \cref{t:convex-comb} yields $C_{i} = C_{k}$,
      as desired.
\end{proof}

\begin{corollary}
  \label{c:lines-pr}
  Let $C$ and $D$
  be nonempty closed convex subsets of $\HH$.
  Then the following are equivalent:
  \begin{enumerate}
    \item 
\begin{math}
  \rbr{\exists \alpha \in \RR \smallsetminus \left\{ 0,1 \right\}} \,
  \rbr{1-\alpha} P_{C} + \alpha P_{D} \in \prH.
\end{math}
\item 
\begin{math}
  \rbr{\forall \alpha \in \RR} ~ \rbr{1-\alpha} P_{C} + \alpha P_{D} 
  \in \prH.
\end{math}
\item 
  $P_{\closu{D-C}}0 \in \rbr{C-C}^{\perp}$
  and 
  $ D=  C + P_{\closu{D-C}}0 $.
  \end{enumerate}
\end{corollary}

\begin{proof}
      This follows from the equivalences 
      ``\cref{i:convex-2}\ensuremath{\iff}\cref{i:convex-3}\ensuremath{\iff}\cref{i:convex-4}''
      of \cref{t:convex-comb}.
\end{proof}

We now specialize \cref{c:lines-pr}
to get a result on scalar multiples of
projectors.

\begin{corollary}
  \label{c:scalar-PC}
  Let $C$ be a nonempty closed
  convex set in $\HH$,
  and let $\alpha \in \RR\smallsetminus\left\{ 0,1 \right\}$.
  Then $\alpha P_{C} \in \prH$
  if and only if $C$ is a singleton. 
\end{corollary}

\begin{proof}
      Let $D = \left\{ 0 \right\}$ in \cref{c:lines-pr}.
\end{proof}


\section{The partial sum 
	property of projectors onto
	convex cones}
\label{sect:partial-sum}

In this section,
we shall discuss
the 
partial sum property
and the
connections
between our work,
Zarantonello's 
\cite[Theorems~5.5 and 5.3]{zarontello1971projections-partI},
and the recent work \cite{bartz2017resolvent}. 
We shall need the following
two results. 
Let us 
provide an instance
where the 
\emph{star-difference}
of two sets (see \cite{hiriart2013convex})
can be explicitly determined. 
\cref{l:star-dif} was mentioned 
in 
\cite[Footnote~5]{bartz2017resolvent} 
and 
was also
stated implicitly
in the proof
of \cite[Theorem~5.2]{zarontello1971projections-partI}.

\begin{lemma}[Star-difference of cones]\label{l:star-dif}
	Let $K_{1}$
	and $K_{2}$
	be nonempty closed convex 
	cones in $\HH$,
	and set
	\begin{equation}
	K \coloneqq \menge{ u \in \HH }{ u + K_{2} \subseteq K_{1} }.
	\end{equation}
	Then the following are 
	equivalent: 
	\begin{enumerate}
		\item\label{i:star-dif1} $K=K_{1}$.
		\item\label{i:star-dif2} $0 \in K$.
		\item\label{i:star-dif3} $K \neq \varnothing$.
		\item\label{i:star-dif4} $K_{2} \subseteq K_{1}$.	
	\end{enumerate}
\end{lemma}

\begin{proof}
	The 
	chain of 
	implications ``\cref{i:star-dif1}\ensuremath{\implies}\cref{i:star-dif2}\ensuremath{\implies}\cref{i:star-dif3}''
	is clear.
	
	``\cref{i:star-dif3}\ensuremath{\implies}\cref{i:star-dif4}'': 
	Fix $u \in K$. 
	Then, since $K_{1}$ and $K_{2}$
	are cones, we infer that
      $\rb{\forall \varepsilon \in \RPP}\,
	\varepsilon u + K_{2} = \varepsilon \rb{u + K_{2}} 
	\subseteq \varepsilon K_{1} = K_{1}.$
	In turn, letting $\varepsilon \downarrow 0$
	and using the closedness of $K_{1}$,
	we obtain $K_{2} \subseteq K_{1}$.
	
	``\cref{i:star-dif4}\ensuremath{\implies}\cref{i:star-dif1}'': 
	First, take 
	$ u \in K_{1}$.
	Since  $K_{2} \subseteq K_{1}$ 
	and $K_{1}$ is a convex cone by
	assumption,
	it follows that
	$u + K_{2} \subseteq K_{1}  + K_{1} \subseteq K_{1} $,
	and therefore $u \in K$. 
	Conversely, fix $u \in K$. 
	Because $u + K_{2} \subseteq K_{1}$
	and $0 \in K_{2}$,
	we deduce that $u \in K_{1}$,
	which completes the proof.
\end{proof}

\begin{proposition}\label{p:compute-sth}
	Let $C$ and $D$
	be nonempty closed convex 
	subsets of $\HH$, 
	and set 
	\begin{equation}\label{e:conju-defn-f}
	f \coloneqq \tfrac{1}{2}d_{C}^{2} + \tfrac{1}{2}d_{D}^{2} - \q
	\quad 
	\text{and}
	\quad 
	h \coloneqq f^{\ast} - \q.
	\end{equation}
	Then the following hold:
	\begin{enumerate}
		\item\label{i:compute1}
			\begin{math}
			\rb{\forall u \in \HH}
                  \,
			h\rb{u}
			= \sup_{v \in D}  \big( \sigma_{C} \rb{u+v} + \scal{u}{v} \big) .
			\end{math}
		\item\label{i:compute2}
		Suppose that $C$ and $D$ are cones 
		and $D \subseteq \pc{C}.$
		Then \begin{math}
			h  = \iota_{\pc{C} \cap \pc{D} }.
		\end{math}
	\end{enumerate}
	
\end{proposition}

\begin{proof}
	\cref{i:compute1}: 
	Since
	$D$ is convex, closed, and nonempty,
	we see that $\iota_{D} \in \varGamma_{0}\rb{\HH}$,
	and so	$\rb{1/2}d_{D}^{2} = \iota_{D} \infconv \q = \iota_{D}\einfconv \q $
	by \cite[Example~12.21
	and Proposition~12.15]{bauschke2017convex}.
	In turn, 
	Moreau's decomposition
	asserts that 
	\begin{math}
	\q - \rb{1/2}d_{D}^{2} = \q - \iota_{D} \einfconv \q 
	= \iota_{D}^{\ast} \einfconv \q.
	\end{math}
	Thus,  
	\cref{e:conju-defn-f} yields
	\begin{equation}
	f = \tfrac{1}{2}d_{C}^{2}  - \iota_{D}^{\ast} \einfconv \q.
	\end{equation}
	Moreover, since $\iota_{D} \in \varGamma_{0}\rb{\HH}$
	and $\q^{\ast} = \q$,
	\cite[Proposition~13.24(i)]{bauschke2017convex}
	and 
	the Fenchel{\textendash}Moreau
	theorem 
	guarantee that 
	\begin{math}
	\rb{\iota_{D}^{\ast} \einfconv \q  }^{\ast} = \iota_{D}^{\ast\ast} + \q^{\ast}
	= \iota_{D} + \q,
	\end{math}
	which implies that $\dom \rb{\iota_{D}^{\ast} \einfconv \q  }^{\ast} = D $. 
	Consequently, because $\iota_{D}^{\ast} \einfconv \q \in \varGamma_{0}\rb{\HH}$,
	\cite[Proposition~14.19 and Example~13.27(iii)]{bauschke2017convex}
	imply that 
	\begin{subequations}
		\begin{align}
		\rb{\forall u \in \HH}
		\quad 
		f^{\ast}\rb{u} - \q\rb{u} 
		& =\rb*{\tfrac{1}{2}d_{C}^{2}  - \iota_{D}^{\ast} \einfconv \q}^{\ast} \rb{u}
		- \q\rb{u} \\
		& = \sup_{v \in \dom \rb{\iota_{D}^{\ast} \einfconv \q  }^{\ast}} 
		\rb*{ \rb*{\tfrac{1}{2}d_{C}^{2} }^{\ast} \rb{u+v}
			-  \rb{\iota_{D}^{\ast}
				\einfconv \q  }^{\ast}\rb{v} } - \q\rb{u} \\
		& = \sup_{v \in D} ~ \Big(\sigma_{C}\rb{u+v} + \q\rb{u+v} 
		- \q\rb{v} \Big)    - \q\rb{u} \\
		& =  \sup_{v \in D} ~ \Big( \sigma_{C} \rb{u+v} + \scal{u}{v} \Big),
		\end{align}
	\end{subequations}
	as announced.
	
	\cref{i:compute2}: First, because 
	$D \subseteq \pc{C}$, 
	\cref{l:star-dif} (applied to 
	the pair of closed convex cones 
	$\rb{\pc{C},D}$)
	yields 
		\begin{equation}\label{e:star-CD}
			\pc{C} = \menge{u \in \HH}{ u + D \subseteq \pc{C}}.
		\end{equation}
	Next, we derive from \cref{i:compute1}
	and \cite[Example~13.3(ii)]{bauschke2017convex}
	that 
		\begin{equation}\label{e:comp-h-cone2}
			\rb{\forall u \in \HH}
			\quad 
			h\rb{u} = \sup_{v \in D}~
			\Big( \sigma_{C}\rb{u+v} + \scal{u}{v}  \Big)
			= \sup_{v \in D} ~
			\Big( \iota_{ C^{\ominus} }\rb{u+v} + \scal{u}{v}  \Big)  .
		\end{equation}
	Now fix $u \in \HH$, and let
	us consider two alternatives.
	
	(a) $u \in \HH\smallsetminus \pc{C}$: 
	In view of \cref{e:star-CD},
	there exists $v \in D$
	such that $u + v \in \HH\smallsetminus \pc{C}$,
	and therefore, 
	by  \cref{e:comp-h-cone2}, 
	\begin{math}
		h\rb{u}
		\geq 
		\iota_{ C^{\ominus} }\rb{u +v} + \scal{u}{v} = +\infty.
	\end{math}

	(b) $u \in \pc{C}$: 
	Then, by  \cref{e:star-CD}, $u + D \subseteq \pc{C}$.
	Hence, 
	since $D$ is a nonempty cone, 
	it follows from 
	\cref{e:comp-h-cone2}
	and \cite[Example~13.3(ii)]{bauschke2017convex}
	that 
	\begin{math}
		h\rb{u} = \sup_{v \in D} \scal{u}{v} 
		= \sigma_{D}\rb{u}
		= \iota_{\pc{D}}\rb{u}
		= \iota_{\pc{C} \cap \pc{D}}\rb{u}.
	\end{math}
	
	Altogether, we
	obtain 
	the desired conclusion. 
\end{proof}

Here is
the first main
result of this section. 
The proof
of the implication 
``\cref{i:sum-pr-cone4}\ensuremath{\implies}\cref{i:sum-pr-cone0}''
was inspired by 
\cite[Lemma~5.3]{bartz2017resolvent}.

\begin{theorem}\label{t:two-cones-pr}
	Let $K_{1}$
	and $K_{2}$ be nonempty closed
	convex cones in $\HH$.
	Then the following are
	equivalent: 
	\begin{enumerate}
		\item \label{i:sum-pr-cone0}
		$K_{1} + K_{2}$
		is closed and 
		$P_{K_{1}} + P_{K_{2}} = P_{K_{1} + K_{2} } .$
		\item \label{i:sum-pr-cone1}
		There exists a nonempty closed
		convex cone $K$ 
		such that $P_{K_{1}} + P_{K_{2}} = P_{K}$.
		\item 
		\label{i:sum-pr-cone2}
		$P_{K_{1}} + P_{K_{2}}$
		is a proximity operator of a
		function in $\varGamma_{0}\rb{\HH}$.
		\item\label{i:sum-pr-cone3}  $P_{K_{1}} + P_{K_{2}}$
		is  nonexpansive.
          \item \label{i:sum-pr-cone4}${\Id} - P_{K_{1}} - P_{K_{2}}$
		is monotone.
		\item\label{i:sum-pr-cone5} \begin{math}
                \rb{\forall x \in \HH}\,
		\scal{P_{K_{1}  } x}{ P_{K_{2}} x} =0.
		\end{math}
	\end{enumerate}
	Furthermore,
	if  one of \cref{i:sum-pr-cone0,i:sum-pr-cone1,i:sum-pr-cone2,i:sum-pr-cone3,i:sum-pr-cone4,i:sum-pr-cone5} holds,
	then 
	\begin{equation}\label{e:dist-cone}
	d_{K_{1} + K_{2} }^{2} = d_{K_{1}}^{2} + d_{K_{2}}^{2} -2\q 
	= d_{K_{1}}^{2} - d_{K_{2}^{\ominus} }^{2} 
	= d_{K_{2}}^{2}- d_{K_{1}^{\ominus}}^{2}.
	\end{equation}
\end{theorem}

\begin{proof}
	The chain of implications
	``\cref{i:sum-pr-cone0}\ensuremath{\implies}\cref{i:sum-pr-cone1}\ensuremath{\implies}\cref{i:sum-pr-cone2}\ensuremath{\implies}\cref{i:sum-pr-cone3}''
	is clear, 
	and
	the implication 
	``\cref{i:sum-pr-cone3}\ensuremath{\implies}\cref{i:sum-pr-cone4}''
	follows from \cite[Example~20.7]{bauschke2017convex}.
	We now assume that 
	\cref{i:sum-pr-cone4} holds
	and establish \cref{i:sum-pr-cone0}. 
	Towards this end, set 
	\begin{equation}
	f \coloneqq 
	\tfrac{1}{2}d_{K_{1}}^{2} 
	+ \tfrac{1}{2}d_{K_{2}}^{2} 
	- \q.
	\end{equation}
	and set  
	\begin{equation}\label{e:cone-defn-h}
	h \coloneqq f^{\ast} -\q.
	\end{equation}
	Let us first establish that 
	\begin{math}
	h = \iota_{K_{1}^{\ominus} \cap K_{2}^{\ominus} }.
	\end{math}
	To do so, we derive from 
      the monotonicity of ${\Id} - P_{K_{1}} - P_{K_{2}}$
	and Moreau's conical decomposition
	that 
	\begin{equation}
	\rb{\forall x \in \HH}
	\quad 
	\scal{ x }{ P_{K_{1}^{\ominus} }x - P_{K_{2}}x }
      = \scal{x - 0}{ \rb{ {\Id} - P_{K_{1}} - P_{K_{2}}}x - 
      \rb{ {\Id} - P_{K_{1}} - P_{K_{2}}}0 }
	\geq 0. 
	\end{equation}
	Thus, because $K_{1}^{\ominus}$
	and $K_{2}$ are closed convex cones, 
	\cite[Lemma~5.6]{zarontello1971projections-partI}
	guarantees that $K_{2} \subseteq K_{1}^{\ominus}$,
	from which and \cref{p:compute-sth}\cref{i:compute2}
	we deduce that 
	\begin{equation}\label{e:form-h}
	h= \iota_{K_{1}^{\ominus} \cap K_{2}^{\ominus} },
	\end{equation}
	as claimed.
	Next, 
	by \cref{t:main} (respectively applied to $P_{K_{1}}$
	and $P_{K_{2}}$), 
	$f$ is Fr\'{e}chet
	differentiable on $\HH$ (hence continuous)
	and 
	\begin{equation}\label{e:cone-f-grad}
        \nabla f = \rb{ {\Id} - P_{K_{1}} } +
        \rb{ {\Id} - P_{ K_{2}}} - {\Id} = {\Id} - P_{K_{1}} - P_{K_{2}},
	\end{equation}
	which is monotone by assumption.
	Therefore, in view of 
	\cite[Proposition~17.7(iii)]{bauschke2017convex}, 
	$f$ is convex, and so $f\in \varGamma_{0}\rb{\HH}$.
	In turn, because  
	$f^{\ast} = h + \q 
	=  \iota_{K_{1}^{\ominus} \cap K_{2}^{\ominus} } + \q$
	by \cref{e:cone-defn-h} and \cref{e:form-h},
	the Fenchel{\textendash}Moreau
	theorem
	and \cite[Example~13.5]{bauschke2017convex}
	yield 
	\begin{math}
	f = f^{\ast\ast } = \rb{ \iota_{K_{1}^{\ominus} \cap K_{2}^{\ominus} } + 
		\q}^{\ast} = \q -\rb{1/2}d_{ K_{1}^{\ominus} \cap K_{2}^{\ominus} }^{2}.
	\end{math}
	Hence, by \cref{e:cone-f-grad}
	and \cite[Corollary~12.31]{bauschke2017convex}, we obtain 
	\begin{math}
        {\Id} - P_{K_{1}} - P_{K_{2}}
	= \nabla f
      = {\Id} - \rb{ {\Id} - P_{K_{1}^{\ominus} \cap K_{2}^{\ominus}}} 
	= P_{K_{1}^{\ominus} \cap K_{2}^{\ominus}}.
	\end{math}
	Thus, 
	the Moreau conical decomposition 
	and \cite[Proposition~6.35 and Corollary~6.34]{bauschke2017convex}
	guarantee that 
	\begin{math}
        P_{K_{1}} + P_{K_{2}} = {\Id} - P_{K_{1}^{\ominus} \cap K_{2}^{\ominus}}
	= P_{ \rb{K_{1}^{\ominus} \cap K_{2}^{\ominus}}^{\ominus} }
	= P_{\closu{ K_{1}^{\ominus\ominus}  + K_{2}^{\ominus\ominus}}}
	= P_{\closu{ K_{1} + K_{2} }}.
	\end{math}
	Consequently, \cref{p:sumCi}\cref{i:set-sum2} 
	asserts that $K_{1} + K_{2}$ is closed,
	and therefore, $P_{K_{1}} + P_{K_{2}}  = P_{K_{1} + K_{2}}$,
	as desired. 
	To summarize, we have shown the
	equivalences of \cref{i:sum-pr-cone0}\textendash\cref{i:sum-pr-cone4}.

	``\cref{i:sum-pr-cone0}\ensuremath{\iff}\cref{i:sum-pr-cone5}'': 
	Follows from \cref{t:sum-pr} and 
	the fact that $\scal{P_{K_{1}}0 }{ P_{K_{2}} 0 } =0$.	
		Moreover, if \cref{i:sum-pr-cone5} holds,
	then \cref{e:dist-cone} follows from \cref{t:sum-pr}
	and \cite[Theorem~6.30(iii)]{bauschke2017convex}.
\end{proof}

Replacing one cone
by a general convex set 
may make 
the implication 
``\cref{i:sum-pr-cone4}\ensuremath{\implies}\cref{i:sum-pr-cone0}''
of \cref{t:two-cones-pr} 
fail, as illustrated by 
the following example.

\begin{example}\label{eg:counter-cone-set}
	Let $K$ be a nonempty closed convex cone 
	in $\HH$, 
	and let $u \in \HH$.
	Then, by
	Moreau's conical decomposition,
      ${\Id} - P_{K} - P_{\left\{ u \right\}} = P_{\pc{K}} - u$,
	which is clearly monotone. However, 
	owing to \cref{eg:uPC},
	$P_{\left\{u\right\}} +  P _{K} = u+ P_{K}$
	is not a projector provided
	that $u \notin \rb{K -K}^{\perp}.$
\end{example}

Here is 
an instance where
the projector
onto the intersection
can be expressed
in term of the 
individual projectors.

\begin{corollary}\label{c:dualized}
	Let $K_{1}$
	and $K_{2}$
	be nonempty closed convex cones
	in $\HH$.
	Then 
	the following 
	are equivalent: 
	\begin{enumerate}
		\item\label{i:dualized1} $P_{K_{1}  \cap K_{2 } } 
              =  P_{K_{1}} + P_{K_{2}}  - {\Id}.$
            \item \label{i:dualized1*} $P_{K_{1}} + P_{K_{2}}  -{\Id}\in \prH$.
          \item\label{i:dualized2} $P_{K_{1}} + P_{K_{2}}  -{\Id} $ is monotone.
		\item\label{i:dualized3} \begin{math}
                \rb{\forall x \in \HH}\,
			\norm{P_{K_{1}}x}^{2} +
			\norm{P_{K_{2}}x}^{2} 
			= \norm{x}^{2} + 
			\scal{P_{K_{1}}x }{ P_{K_{2}}x }.
		\end{math}		
	\end{enumerate}
\end{corollary}

\begin{proof}
	We first deduce from  
	the Moreau conical decomposition
	and 
	\cite[Proposition~6.35]{bauschke2017convex}
	that 
	\begin{subequations}
		\begin{align}
		P_{K_{1}  \cap K_{2 } } =  P_{K_{1}} + P_{K_{2}} 
            -{\Id}
            & \iff {\Id} - P_{\rb{K_{1} \cap K_{2}}^{\ominus} } 
            = P_{K_{1}} + P_{K_{2}} - {\Id} \\
            & \iff {\Id} - P_{ \closu{ K_{1}^{\ominus} + K_{2}^{\ominus} } }
            = P_{K_{1}} + P_{K_{2}}  - {\Id} \\
		& \iff P_{ \closu{ K_{1}^{\ominus} + K_{2}^{\ominus} } } = 
            \rb{ {\Id} - P_{K_{1}}} + \rb{ {\Id} - P_{K_{2}}} \\
		& \iff  P_{ \closu{ K_{1}^{\ominus} + K_{2}^{\ominus} } } 
		= P_{K_{1}^{\ominus} } + P_{K_{2}^{\ominus} }. \label{e:intersect-equivalence} 
		\end{align}
	\end{subequations}
	
	``\cref{i:dualized1}\ensuremath{\iff}\cref{i:dualized1*}'': 
	Denote by $\mathscr{C}$
	the class of nonempty closed convex 
	cones in $\HH$. 
	Then,
	because the
	mapping $\mathscr{C} \to \mathscr{C} : K \mapsto \pc{K}$
	is bijective due to 
	\cite[Corollary~6.34]{bauschke2017convex}, 
	we  
	derive
	from \cref{e:intersect-equivalence}, the equivalence ``\cref{i:sum-pr-cone0}\ensuremath{\iff}\cref{i:sum-pr-cone1}''
	of \cref{t:two-cones-pr},
	and the Moreau conical decomposition that 
	\begin{subequations}
		\begin{align}
		\cref{i:dualized1} 
			& \iff P_{ \closu{ K_{1}^{\ominus} + K_{2}^{\ominus} } } 
			= P_{K_{1}^{\ominus} } + P_{K_{2}^{\ominus} } \\
			& \iff \text{$K_{1}^{\ominus} + K_{2}^{\ominus}$~is closed and~}
			P_{  K_{1}^{\ominus} + K_{2}^{\ominus} } 
			= P_{K_{1}^{\ominus} } + P_{K_{2}^{\ominus} } \\
                  & \iff \rb{\exists K \in \mathscr{C}}\, P_{K} 
			 = P_{K_{1}^{\ominus} } + P_{K_{2}^{\ominus} } \\
                   & \iff \rb{\exists K \in \mathscr{C}}\, 
                  {\Id} - P_{\pc{K}} =  
                  \rb{ {\Id} - P_{K_{1}}} + \rb{ {\Id} - P_{K_{2}}} \\
                  & \iff \rb{\exists K \in \mathscr{C} }\,
                  P_{\pc{K} } = P_{K_{1}} + P_{K_{2}} - {\Id} \\
                  & \iff \rb{\exists K \in \mathscr{C} }\,
			 P_{ K } = P_{K_{1}} + P_{K_{2}} - \Id \\
			 & \iff \cref{i:dualized1*},
		\end{align}
	\end{subequations}
	where the last equivalence follows from
	the fact that $P_{K_{1}} + P_{K_{2}} - \Id$
	is positively homogeneous.
		
	``\cref{i:dualized1}\ensuremath{\iff}\cref{i:dualized2}'': 
	Since $ \Id - \rb{P_{K_{1}^{\ominus} } + P_{K_{2}^{\ominus} }} = 
		P_{K_{1}} + P_{K_{2}} - \Id$
		by Moreau's decomposition, 
		this equivalence 
		is
	a consequence of \cref{e:intersect-equivalence} 
	and
	 the equivalence 
	 ``\cref{i:sum-pr-cone1}\ensuremath{\iff}\cref{i:sum-pr-cone4}''
	of \cref{t:two-cones-pr} 
	(applied to $\rb{K_{1}^{\ominus}, K_{2}^{\ominus} }$).

	``\cref{i:dualized1}\ensuremath{\iff}\cref{i:dualized3}'': 
	This readily follows from 
	\cref{e:intersect-equivalence},  
	the equivalence 
	``\cref{i:sum-pr-cone1}\ensuremath{\iff}\cref{i:sum-pr-cone5}''
	of \cref{t:two-cones-pr} (applied to 
	$\rb{K_{1}^{\ominus},K_{2}^{\ominus} }$), 
	and 
	\cref{l:cone-identities}\cref{i:cone-identities2}.
\end{proof}

By replacing $\rb{K_{1}, K_{2}}$ by 
$\rb{K_{1}, K_{2}^{\ominus}}$ in \cref{c:dualized},
we provide an alternative proof
for \cite[Theorem~5.3]{zarontello1971projections-partI}.
The linear case
of \cref{c:Zaran-difference}
goes back  
at least to
Halmos 
(see \cite[Theorem~3, p.~48]{Halmos-spectral-1951}).

\begin{corollary}[Zarantonello]
	\label{c:Zaran-difference}
	Let $K_{1}$
	and $K_{2}$
	be nonempty closed convex cones 
	in $\HH$. 
	Then $P_{K_{2}} P_{K_{1}} = P_{K_{2}}$
	if and only if $P_{K_{1}} - P_{ K_{2} }$
	is a projector onto a closed
	convex set;
	in which case,  
	$P_{K_{1}} - P_{K_{2}} = P_{K_{1} \cap K_{2}^{\ominus} } .$
\end{corollary}

\begin{proof}
	First, suppose that $P_{K_{2}} P_{K_{1}} = P_{K_{2}}. $ 
	Then, 
	by \cite[Theorem~6.30(i)\&(iii)]{bauschke2017convex}
	and \cref{l:cone-identities}\cref{i:cone-identities1}, 
		\begin{subequations}
			\begin{align}
				\rb{\forall x \in \HH}
				\quad 
				\scal{P_{K_{1}}x }{ P_{K_{2}^{\ominus} } x  }
				 = \scal{ P_{K_{1}}x }{x} - \scal{ P_{K_{1}}x }{ P_{K_{2}}x } 
				& = \scal{ P_{K_{1}}x }{x} - 
					\scal{ P_{K_{1}}x }{ P_{K_{2}} P_{K_{1}} x } \\
				& = \norm{P_{K_{1}} x}^{2} - \norm{ P_{K_{2}} P_{K_{1}} x}^{2} \\
				& = \norm{P_{K_{1}} x}^{2} - \norm{ P_{K_{2}} x}^{2}  \\
				& = \norm{ P_{K_{1}}x }^{2} + \norm{P_{K_{2}^{\ominus} }x }^{2} 
					- \norm{x}^{2}.
 			\end{align}
		\end{subequations}
	Hence, 
	the equivalence ``\cref{i:dualized1}\ensuremath{\iff}\cref{i:dualized3}'' 
	of 
	\cref{c:dualized}
	(applied to $\rb{K_{1} , K_{2}^{\ominus} }$)
	yields $P_{K_{1}} - P_{K_{2}} 
	= P_{K_{1}} + P_{K_{2}^{\ominus} } - \Id 
	= P_{K_{1} \cap K_{2}^{\ominus} } $,
	as desired.
	Conversely, assume that $P_{K_{1}} - P_{K_{2}}$
	is a projector associated with a 
	closed convex set. 
	Since $P_{K_{1}} - P_{K_{2}} = P_{K_{1}} + P_{K_{2}^{\ominus}} - \Id $,
	it follows from 
	the equivalence ``\cref{i:dualized1}\ensuremath{\iff}\cref{i:dualized1*}''
	of 
	\cref{c:dualized} (applied to $\rb{K_{1}, K_{2}^{\ominus} }$)
	that  
		\begin{equation}\label{e:Zaran-PcPd}
			P_{K_{1}} - P_{K_{2}}  = P_{K_{1} \cap K_{2}^{\ominus}  }.
		\end{equation}
	Now take $x \in \HH$. On the one hand,
	because $P_{K_{1} \cap K_{2}^{\ominus} } + P_{K_{2}} 
	=  \rb{P_{K_{1}}  - P_{K_{2}} } + P_{K_{2}} = P_{K_{1}}$
	by \cref{e:Zaran-PcPd},
	we infer from \cref{t:two-cones-pr} 
	that 
		\begin{math}
			P_{K_{1} \cap K_{2}^{\ominus} }x
			\perp P_{K_{2}}x 
 		\end{math}
 	or, equivalently, by \cref{e:Zaran-PcPd}, 
 		\begin{math}
 		 \rb{ P_{K_{1}}x - P_{K_{2}}x }
 		\perp P_{K_{2}}x .
 		\end{math}
 	On the other hand, \cref{e:Zaran-PcPd} implies
 	that $ P_{K_{1}}x - P_{K_{2}}x \in K_{2}^{\ominus} $.
 	Altogether, 
 	since clearly $P_{K_{2}}x \in K_{2}$,
 	\cite[Proposition~6.28]{bauschke2017convex}
 	asserts that $P_{K_{2}} P_{K_{1}}x = P_{K_{2}}x $,
 	and the proof is complete.
\end{proof}

The so-called partial
sum property, i.e., 
if a 
finite sum 
of proximity operators
is a proximity operator, 
then so is
every partial sum, 
was obtained in \cite{bartz2017resolvent}.
Somewhat surprisingly, 
as we shall see in 
the following result,
this property is still
valid
in the class
of projectors onto 
convex cones. 
The equivalence ``\cref{i:cone-many1}\ensuremath{\iff}\cref{i:cone-many2}''
of the following result
was obtained by Zarantonello
with a different proof
(see \cite[Theorem~5.5]{zarontello1971projections-partI}).

\begin{theorem}[Partial sum property for cones]\label{c:cone-many}
      Let
	$\fa{K_{i}}{ i \in I }$
	be a family of nonempty 
	closed convex cones in $\HH$.
	Then the following
	are equivalent: 
	\begin{enumerate}
		\item \label{i:cone-many1}For every $\rb{i,j} \in I\times I$
		such that $i \neq j$, we have 
		\begin{math}
              \rb{\forall x \in \HH}\,
		\scal{P_{K_{i}  } x}{ P_{K_{j}} x} =0.
		\end{math}
		\item \label{i:cone-many1prime}
		$\sum_{ i \in I}K_{i}$ is closed
		and $\sum_{ i \in I} P_{K_{i}} = P_{\sum_{ i \in I} K_{i} } $.
		\item\label{i:cone-many2} $\sum_{ i \in I} P_{K_{i}}$
		is a projection onto
		a closed convex cone in $\HH$.
		\item\label{i:cone-many3} $\sum_{ i \in I} P_{K_{i}}$
		is a proximity operator of
		a function in $\varGamma_{0}\rb{\HH}$.
		\item \label{i:cone-many4}For every nonempty subset $J$
		of $I$, $\sum_{j \in J} P_{ K_{j}  }$
		is a proximity operator of
		a function in $\varGamma_{0}\rb{\HH}$.
		\item \label{i:cone-many4a} 
		For every nonempty subset $J$
		of $I$, $\sum_{ j \in J}K_{j}$ is closed
		and $\sum_{ j \in J} P_{K_{j}} = P_{\sum_{ j \in J} K_{j} } $.
		\item\label{i:cone-many5} For every $\rb{i,j} \in I\times I$
		such that $i \neq j$,
		we have   $P_{ K_{i} } + P_{ K_{j} }$
		is nonexpansive.
		\item\label{i:cone-many6} For every $\rb{i,j} \in I\times I$
		such that $i \neq j$,
		we have 
		\begin{math}
		\Id - P_{K_{i}} - P_{K_{j}}
		\end{math}
		is monotone.
	\end{enumerate}
\end{theorem}

\begin{proof}
	``\cref{i:cone-many1}\ensuremath{\implies}\cref{i:cone-many1prime}'':
	A direct consequence of \cref{c:induct-many}.
	
	``\cref{i:cone-many1prime}\ensuremath{\implies}\cref{i:cone-many2}'' and 
	``\cref{i:cone-many2}\ensuremath{\implies}\cref{i:cone-many3}'': 
	Clear.
	
	``\cref{i:cone-many3}\ensuremath{\implies}\cref{i:cone-many4}'':
	Let $f \in \varGamma_{0}\rb{\HH}$
	be such that $\sum_{ i \in I} P_{ K_{i} } = \prox_{f}$.
	Then, by Moreau's decomposition
      (\cite{Moreau-prox-dual-1965}),
	$\sum_{i \in I}P_{K_{i}} + \prox_{f^{\ast}} = \prox_{f} + \prox_{f^{\ast}} = \Id$.
	Therefore, since $\set{P_{ K_{i} }}{ i \in I}$
	are proximity operators, the 
	conclusion follows from \cite[Theorem~4.2]{bartz2017resolvent}.
	
	``\cref{i:cone-many4}\ensuremath{\implies}\cref{i:cone-many5}'': Clear.
	
	``\cref{i:cone-many5}\ensuremath{\implies}\cref{i:cone-many6}'': 
	See \cite[Example~20.7]{bauschke2017convex}.
	
	``\cref{i:cone-many6}\ensuremath{\implies}\cref{i:cone-many1}'': 
	This is the implication 
	``\cref{i:sum-pr-cone4}\ensuremath{\implies}\cref{i:sum-pr-cone5}''
	of
	\cref{t:two-cones-pr}. 
	
	To sum up, we have shown the equivalence
	of \cref{i:cone-many1}{\textendash}\cref{i:cone-many6}
	except for \cref{i:cone-many4a}.
	
	``\cref{i:cone-many4}\ensuremath{\iff}\cref{i:cone-many4a}'': 
	Follows from the equivalence 
	``\cref{i:cone-many1prime}\ensuremath{\iff}\cref{i:cone-many3}.''
\end{proof}

As we now illustrate,
the partial sum property may, however,
fail
outside the class 
of projectors onto convex cones.

\begin{example}\label{eg:two-rays2}
	Suppose that $\HH \neq \{0\}$, 
	let $w \in \HH\smallsetminus\{0\}$, 
	set $U \coloneqq \RP w$,
	and set $V \coloneqq  \RP\rb{-w} =\RM w .$
	Then, appealing to 
	\cite[Example~29.31]{bauschke2017convex},
      we see that $\rb{\forall x \in \HH}\, 
	\scal{P_{U}x }{ P_{V} x } = 0$.
	Hence, by \cref{t:two-cones-pr},
	$P_{U} + P_{V} = P_{U+V} = P_{\RR w}.$
	Now suppose that $z \in \HH \smallsetminus \rb{U - U}^{\perp} 
	=\HH \smallsetminus \rb{\RR w}^{\perp}$.
	Then clearly 
	$P_{U} + P_{V} + P_{\left\{z\right\}}
	+ P_{\left\{-z\right\}} = P_{\RR w}$
	is the projector
	associated with 
	the line $\RR w$.
	However, due to \cref{eg:uPC},
	$P_{\left\{z\right\}} + P_{U} = z+ P_{U}$
	is not a projector.
\end{example}

To proceed further,
we  require the following lemma.

\begin{lemma}
      \label{l:2cones-rays}
      Let $u$ and $v$ be in $\HH \setminus \left\{ 0 \right\}$,
      and set $U \coloneqq \RP u $ and $V \coloneqq
      \RP v$.
      Then 
      \begin{equation}
            \label{eq:2rays}
            \left[ \;  \rbr{\forall x \in \HH}\,
            \scal[\big]{P_{U}x}{P_{V}x} =0 \;\right]
      \quad \iff \quad 
      \left[ \; u \in \RMM v \text{ or } u\perp v \; \right].
      \end{equation}
\end{lemma}

\begin{proof}
      Suppose first that
      \begin{equation}
            \label{eq:suppose-first}
      \rbr{\forall x \in \HH} \, 
      \scal{P_{U}x}{P_{V}x} = 0,
      \end{equation}
      and set $w \coloneqq \norm{v}u + \norm{u}v$.
      Then, by the Cauchy--Schwarz 
      inequality,
      $\scal{w}{u} = \norm{v}\norm{u}^{2} + 
      \norm{u} \scal{u}{v}
      \geq \norm{v}\norm{u}^{2} - 
      \norm{u}\rbr{\norm{u}\norm{v}} = 0$
      and 
      $\scal{w}{v} = \norm{v}\scal{u}{v} + \norm{u}\norm{v}^{2} 
      \geq - \norm{v} \rbr{\norm{u} \norm{v} }
      + \norm{u} \norm{v}^{2} = 0$.
      Hence, due to \cite[Example~29.31]{bauschke2017convex},
      we obtain
      \begin{equation}
            \label{eq:projection-cones}
            P_{U}w = \frac{\scal{w}{u}}{\norm{u}^{2}}u
            = \frac{ \norm{v}\norm{u}^{2} + \norm{u}\scal{u}{v} }{\norm{u}^{2}}u
            \quad
            \text{and}
            \quad 
            P_{V}w = \frac{\scal{w}{v}}{\norm{v}^{2}}v
            = \frac{\norm{v}\scal{u}{v} + \norm{u}\norm{v}^{2}}{\norm{v}^{2}}v.
      \end{equation}
      In turn, it follows from \cref{eq:suppose-first}
      that 
      \begin{equation}
            0 = \scal[\big]{P_{U}w}{P_{V}w}
            = \frac{1}{\norm{u}^{2}\norm{v}^{2}}
            \rbr[\big]{\norm{v}\norm{u}^{2} + \norm{u}\scal{u}{v}}
            \rbr[\big]{\norm{v}\scal{u}{v}+\norm{u}\norm{v}^{2}}\scal{u}{v},
      \end{equation}
     from which we derive
     the 
     following conceivable cases.
      
     (a) $\scal{u}{v} = 0$: Then $u \perp v$.

     (b) $\norm{v}\norm{u}^{2} + \norm{u}\scal{u}{v} = 0 $:
     Then, since $\norm{u} > 0$,
     we get $\norm{v}\norm{u} + \scal{u}{v} = 0$ or, equivalently, 
     $\scal{u}{-v} = \norm{v}\norm{u} = \norm{-v}\norm{u}$.
     Consequently, the Cauchy--Schwarz inequality
     yields $u \in \RPP \rbr{-v} = \RMM v$.

     (c) $\norm{v}\scal{u}{v} + \norm{u}\norm{v}^{2} = 0$:
     Proceeding as in the case (b), we obtain
     $u \in \RMM v$. 

     Conversely, assume that 
     $u \in \RMM v$ or that $u \perp v$.
     Let $x \in \HH$.
     If $u \in \RMM v$,
     then $V = \RP v = \RP \rbr{-u}$,
     and since $U = \RP u$,
      it follows from \cite[Example~29.31]{bauschke2017convex}
      that $0 \in \left\{ P_{U}x, P_{V}x \right\}$ and so
      $ \scal{P_{U}x}{P_{V}x} = 0$.
     Otherwise, we have $\scal{u}{v}=0$ and,
     invoking \cite[Example~29.31]{bauschke2017convex} once more,
     \begin{equation}
           \scal{P_{U}x}{P_{V}x} = 
           \frac{1}{\norm{u}^{2}\norm{v}^{2}}\max\left\{\scal{x}{u},0\right\}
           \max\left\{
           \scal{x}{v},0\right\}\scal{u}{v} = 0,
     \end{equation}
     as required.
\end{proof}

\cref{c:cone-many}
allows us to characterize
finitely generated cones
of which the associated
projectors are the sum
of projectors onto
the generating rays.

\begin{proposition}
      \label{p:finite-cones}
      Let $\fa{u_{i}}{i \in  I}$
      be a finite family in $\HH\setminus\left\{ 0 \right\}$,
       set 
       $\rbr{\forall i \in I} \, K_{i} \coloneqq \RP u_{i}$, 
       and set 
       $K \coloneqq \sum_{i \in I} K_{i}$.
      Then
      $P_{K} = \sum_{i \in I} P_{K_{i}}$
      if and only if,
       for every $\rbr{i,j} \in I \times I$
      with $i \neq j$,
      we have $u_{i} \perp u_{j}$
      or $u_{j} \in \RMM u_{i}$;
      in which case,
      for every $i \in I$,
      $\card \menge{j \in I}{u_{j} \in \RMM u_{i}} \leq 1$.
\end{proposition}

\begin{proof}
      Assume first 
      that $P_{K} = \sum_{i \in I} P_{K_{i}}$.
      Then, \cref{c:cone-many} ensures that,
      \begin{equation}
            \label{eq:ensures}
            \text{
      for every $\rbr{i,j} \in I \times I$
      with $i \neq j$, we have $\rbr{\forall x \in \HH } \, 
\scal{P_{K_{i}}x}{P_{K_{j}}x} = 0$.}
      \end{equation}
      Hence, 
      for every $\rbr{i,j} \in I \times I$
      with $i \neq j$, because $ K_{i} = \RP u_{i}$
      and $K_{j} = \RP u_{j}$,
      we derive from \cref{l:2cones-rays}
      that $u_{j} \in \RMM u_{i}$ 
      or $u_{i} \perp u_{j}$.
      Now fix $i \in I$, 
      and let us show that 
      $\card \menge{j \in I}{u_{j} \in \RMM u_{i}} \leq 1$
      by contradiction: 
      suppose that there exists $\rbr{j,k} \in I \times I$
      such that $j \neq k$ and 
    $\left\{ u_{j},u_{k} \right\} \subset \RMM u_{i}$.
    Then,  clearly $u_{j} \in \RPP u_{k}$.
      On the other hand,
      \cref{eq:ensures} and \cref{l:2cones-rays}
      imply that $u_{j} \perp u_{k}$
      or $u_{j} \in \RMM u_{k}$.
      Altogether, we reach
      a contradiction.
      To sum up,  $\card\menge{j \in I}{ u_{j } \in \RMM u_{i}} 
      \leq 1$.
      To see the converse,
      combine \cref{l:2cones-rays} 
      and \cref{c:cone-many}.
\end{proof}


\section{The one-dimensional  case}
\label{sect:R}

In this section, we assume that 
\begin{empheq}[box = \mybluebox]{equation}
\HH = \RR 
\end{empheq}
and that 
\begin{empheq}[box = \mybluebox]{equation}
\text{$C$ and $D$
are nonempty closed convex subsets,
i.e., closed intervals, 
of $\HH$.}
\end{empheq}
The goal
of this section is to
describe all pairs
 $\rb{C,D}$
 on the real line 
such that $P_{C} + P_{D} = P_{C+D}.$ 
We begin with 
a simple observation.

\begin{remark}
	If $C = \left\{0\right\}$ or $D = \left\{0\right\}$,
	then clearly $P_{C} + P_{D} = P_{C+D}$.
	Thus, we henceforth assume in this section 
	that 
	 \begin{empheq}[box = \mybluebox]{equation}
	 \label{e:CD-nonzero}
	 \text{$C \neq \left\{0\right\}$ and $D \neq \left\{0\right\}$.
	}
	 \end{empheq}
\end{remark}

Here
is a sufficient condition
under which $P_{C} + P_{D} = P_{C+D}.$

\begin{proposition}
	\label{p:CD-R-intersect}
	Suppose that 
	$C \cap D = \left\{0\right\}.$
	Then the following hold:
	\begin{enumerate}
		\item\label{i:CD-inter-1} Exactly
		one of the following
		cases occurs: 
		\begin{enumerate}
			\item\label{i:CD-R-intersect1} $C \subseteq \RM$. 
			Then $D \subseteq \RP$
			and  $\max{C} = \min{D} = 0.$
			\item\label{i:CD-R-intersect2} 
			\begin{math}
				C \cap \RPP \neq \varnothing.
			\end{math}
			Then $C \subseteq \RP$, $D \subseteq \RM$, 
			and 
			$\min{C} = \max {D} = 0.$
			
		\end{enumerate}
	\item \label{i:CD-inter-2} 
	$C+D$ is closed and 
	 $P_{C} + P_{D} = P_{C+D}.$
	\end{enumerate}
\end{proposition}

\begin{proof}
	\cref{i:CD-R-intersect1}: Since $C \subseteq \RM$ 
	and $0 \in C$, it follows that $\max{C} = 0.$ 
	Let us now verify that 
	$D \subseteq \RP.$ 
	Assume to the contrary that 
	there exists $\xi \in D \cap \RMM.$ 
	Then,
	on the one hand, 
	by the convexity of $D$
	and the fact that $0 \in D$,
	it follows that 
	$\sqbc{\xi ,0} \subseteq D$.
	On the other hand, 
	because $\left\{0\right\} \neq C \subseteq \RM$
	and $C$ is convex,
	there exists $\eta \in \RMM$
	satisfying $\sqbc{\eta , 0} \subseteq C$. 
	Altogether, $\sqbc{ \max\left\{ \xi, \eta \right\} , 0 } \subseteq C \cap D$,
	and we reach a contradiction. 
	Thus, $D \subseteq \RP$, and since $0 \in D$,
	we conclude that $\min{D} = 0$, as desired.
	
	\cref{i:CD-R-intersect2}: 
	We first argue by contradiction that  $D \subseteq \RM$.
	Towards this goal, assume that 
	there exists $\eta \in D \cap \RPP,$
	and take $\xi \in C \cap \RPP.$ 
	Since $0 \in C$ and $\xi \in C$,
	the convexity of $C$ yields 
	$\sqbc{0,\xi} \subseteq C$,
	and likewise, $\sqbc{0,\eta }\subseteq D$.
	Thus, $\sqbc{0,\min\left\{ \xi,\eta \right\}} \subseteq C \cap D$,
	which contradicts our assumption.
	Hence, $D \subseteq \RM$. 
	Consequently, 
	by interchanging the roles of $C$
	and $D$ in \cref{i:CD-R-intersect1}, 
	we obtain the conclusion.
	
	\cref{i:CD-inter-2}: 
	In the light of \cref{i:CD-inter-1},
	we may and do assume that 
	$C \subseteq \RM.$
	Then \cref{i:CD-R-intersect1}
	implies that $\max{C} = \min{D} = 0$,
	and 
	thus \cite[Example~24.34(i)]{bauschke2017convex}
	yields  	
			\begin{math}
				\rb{\forall \xi \in \HH}
                        \,
				\scal{P_{C}\xi }{ P_{D}\xi }
			 = 0.
			\end{math}
	Hence, according to \cref{t:sum-pr},
	we conclude that 
	$C+D$ is closed and that 
	$P_{C} + P_{D} = P_{C+D}$.
\end{proof}

Here is a
direct consequence of \cref{p:CD-R-intersect}\cref{i:CD-inter-1}.

\begin{corollary}
	Suppose that $C \cap D \neq \varnothing.$
	Then 
		\begin{equation}
			C \cap D = \left\{0 \right\}
			\quad 
			\iff 
			\quad 
			CD \subseteq \RM,
		\end{equation}
	where $CD \coloneqq \menge{\xi \eta}{
	\xi \in C\text{~and~}\eta \in D}.$
\end{corollary}

The next result
classifies
all pairs
$\rb{C,D}$
such that $P_{C} + P_{D} = P_{C+D}.$ 
Item \cref{i:CD-R-12} is a
partial converse of \cref{p:CD-R-intersect}.

\begin{theorem}[Dichotomy]\label{p:CD-R-1}
	Suppose that $P_{C} + P_{D} \in \prH.$
	Then exactly
	one of the following
	cases occurs: 
	\begin{enumerate}
		\item\label{i:CD-R-11} 
		$C$ and $D$ are singletons.
		\item \label{i:CD-R-12}
		Neither $C$ nor
		$D$ is a singleton and 
		$C \cap D =\left\{0\right\}.$
	\end{enumerate}
\end{theorem}

\begin{proof}
	\cref{t:sum-pr} and
	our assumption 
	guarantee 
	the existence of $\gamma \in \RR$
	such that 
		\begin{equation}\label{e:CD-R-gamma}
			\rb{\forall \xi \in \HH}
			\quad 
			\rb{P_{C}\xi}\rb{P_{D}\xi} =  \scal{ P_{C} \xi }{ P_{D} \xi } =
			\gamma.
		\end{equation}
	
	\cref{i:CD-R-11}: 
	Suppose that $C  = \left\{\omega  \right\}$,
	where $\omega  \neq 0 $
	due to \cref{e:CD-nonzero}. 
	Then, for every $\xi \in D$ and every $\eta \in D$,
	since $P_{C}\xi = P_{C}\eta = \omega$, 
	\cref{e:CD-R-gamma} implies that 
	\begin{math}
		\omega \xi = \omega P_{D}\xi = \omega P_{D}\eta = \omega \eta,
	\end{math}
	and because $\omega \neq 0$, 
	it follows that $\xi = \eta$.
	Therefore, $D$ is a singleton, as required.
	
	\cref{i:CD-R-12}: 
	Suppose that $C$
	is not a singleton.
	Let us first show that 
	$D$ is not a singleton
	by proving the contrapositive.
	If $D$ is a singleton,
	then interchanging 
	$C$
	and $D$ in \cref{i:CD-R-11},
	we see that $C$ is a singleton. 
	Next, we shall establish that 
		$C \cap D \neq \varnothing$
	by contradiction.
	Assume that 
	$C \cap D = \varnothing.$
	Then, owing to
	\cite[Theorem~3.53]{bauschke2017convex},
	we obtain $\mu \in \HH \smallsetminus \left\{0\right\}$
	and $\beta \in \RR$
	such that 
	\begin{equation}\label{e:R-CD-separate}
		\rb{\forall \xi \in C}\rb{\forall \eta \in D}
		\quad 
		\xi \mu \leq \beta \leq  \eta \mu .
	\end{equation}
	Without loss of generality, 
	we may and do assume that 
	\begin{equation}\label{e:mu-positive}
		\mu >0.
	\end{equation}
	Then \cref{e:R-CD-separate} asserts
	that $C$ is bounded above
	and $D$ is bounded below,
	and because they are closed,
	we infer that $\sup{C} = \max{C}$
	and $\inf{D} = \min{D}.$
	Let us consider the following
	conceivable cases.
	
	(a) $\max{C} = 0$:  Then $\min{D} \neq 0$
	(otherwise $ 0 \in C \cap D$, which is absurd).
	Because $C$ is not a singleton,
	we can find $\xi_{1} \in C$ and $\xi_{2} \in C$
	such that  $\xi_{1} \neq \xi_{2}$.
	In turn, due to \cref{e:R-CD-separate}
	and \cref{e:mu-positive},
	\begin{math}
		\rb{\forall i \in \left\{1,2\right\}}
            \,
		\xi_{i} \leq \beta / \mu \leq \min{D},
	\end{math}
	from which and \cite[Example~24.34(i)]{bauschke2017convex}
	we deduce that 
	$P_{D}\xi_{1} = P_{D}\xi_{2} = \min{D}$.
	Consequently, 
	since $\left\{ \xi_{1}, \xi_{2} \right\} \subseteq C$, 
	\cref{e:CD-R-gamma}
	implies that $\xi_{1} \min{D} 
	= \scal{ P_{C}\xi_{1} }{ P_{D}\xi_{1} } 
	=  \scal{ P_{C}\xi_{2} }{ P_{D}\xi_{2} }
	= \xi_{2} \min{D}$,
	and since $\min{D} \neq 0$,
	it follows that $\xi_{1} = \xi_{2}$,
	which is impossible.
	
	(b) $\max{C} \neq 0$: Since $D$
	is not a singleton, there are $\eta_{1} \in D$
	and $\eta_{2} \in D$ such that $\eta_{1} \neq\eta_{2}.$
	In turn, we  infer from \cref{e:R-CD-separate}\&\cref{e:mu-positive}
	 that
       $\rb{\forall i \in \left\{1,2\right\}}\, \max{C} \leq \beta/\mu 
	\leq \eta_{i}, $
	and therefore \cite[Example~24.34(i)]{bauschke2017convex}
	yields $P_{C}\eta_{1} = P_{C}\eta_{2} = \max{C}.$
	Thus, by \cref{e:CD-R-gamma}
	and the fact that 
	$\left\{ \eta_{1}, \eta_{2} \right\} \subseteq D$, , 
	we see that $  \rb{\max{C} }\eta_{1}
	=
	\scal{ P_{C}\eta_{1} }{ P_{D}\eta_{1} } 
	=  \scal{ P_{C}\eta_{2} }{ P_{D}\eta_{2} }
	 =  \rb{\max{C} }\eta_{2}$.
	Consequently, since $\max{C} \neq 0$,
	it follows that $\eta_{1} = \eta_{2}$,
	which is absurd.
	
	To summarize, we have shown that 
		\begin{equation}
			C \cap D \neq \varnothing.
		\end{equation}
	Let us next verify that 
	$C \cap D$ is a singleton.
	To this end, 
	take $\xi\in C \cap D$
	and  $\eta \in C \cap D$,
	and let $\varepsilon \in \sqbo{0,1}.$
	On the one hand,
	by \cref{e:CD-R-gamma},
	we see that 
	\begin{equation}\label{e:CD-singleton}
		\rb{\forall \xi \in C \cap D}
		\quad 
		 \xi^{2} = 
		\gamma.
	\end{equation}
	On the other hand, 
	since $C\cap D$
	is convex
	and $\varepsilon \in \sqbo{0,1}$,
	$\rb{1-\varepsilon}\xi + \varepsilon \eta \in C\cap D$.
	Altogether, 
	\begin{math}
		\xi^{2} = \sqbc{\rb{1-\varepsilon}\xi + \varepsilon \eta}^{2} 
	\end{math}
	or, equivalently, 
	\begin{math}
		\varepsilon \rb{\xi - \eta}\sqbc{ \rb{2-\varepsilon}\xi + \varepsilon \eta }
		= \xi^{2} - \sqbc{\rb{1-\varepsilon}\xi + \varepsilon \eta}^{2} 
	=0.
	\end{math}
	Interchanging $\xi$ and $\eta$
	yields
	$\varepsilon \rb{\eta - \xi}\sqbc{ \rb{2-\varepsilon}\eta + \varepsilon \xi }
	=0$,
	and upon adding these equalities,
	we obtain 
	\begin{math}
		\varepsilon \rb{\xi -\eta}\rb{2\rb{1-\varepsilon}\xi - 2\rb{1-\varepsilon}\eta} =0,
	\end{math}
	i.e.,
	\begin{math}
		2\varepsilon \rb{1-\varepsilon}\rb{\xi -\eta}^{2} = 0.
	\end{math}
	Therefore, $\xi =\eta$ and $C \cap D$ is thus a singleton,
	say 
		\begin{equation}
			C \cap D =\left\{ \omega \right\}.
		\end{equation}
	It remains to show that 
	$\omega = 0.$ 
	Since $C$ is not a singleton,
	there exists $\xi \in C \smallsetminus\left\{ \omega \right\}$.
	In turn, 
	because $C\cap D = \left\{ \omega \right\}$, 
	we derive from  
	\cite[Proposition~24.47]{bauschke2017convex}
	(applied to $\rb{\varOmega, \phi} = \rb{D,\iota_{C}}$)
	and \cref{e:CD-R-gamma}\&\cref{e:CD-singleton} that
	 $\xi \omega = \scal{ P_{C}\xi }{ P_{C\cap D}\xi }
	 = \scal{ P_{C}\xi }{ P_{D}\rb{P_{C}\xi} }
	 = \scal{ P_{C}\xi }{ P_{D}\xi  }
	 = \gamma
	 = \omega^{2}.$
	 Thus, $\omega \rb{\xi - \omega} = 0$,
	 and since $\xi \neq \omega$,
	 it follows that $\omega = 0$,
	 which completes the proof.
\end{proof}


\section{On a result by Halmos}
\label{sect:more}

In this section,
we revisit
and extend
the classical result 
\cite[Theorem~2, p.~46]{Halmos-spectral-1951}
to the nonlinear case.

\begin{proposition}\label{c:set-cone}
	Let $C$ be
	a nonempty closed convex subset of $\HH$,
	and let $K$ be a
	nonempty closed convex cone in $\HH$.
	Suppose that 
	there exits a closed convex set $D$ such that
	$P_{C} + P_{K} = P_{D}$.
	Then $C \subseteq \pc{K}.$
\end{proposition}

\begin{proof}
	Our assumption and \cref{t:sum-pr}
	guarantee the existence of $\gamma \in \RR$
      such that $\rb{\forall x \in \HH}\,\scal{P_{C}x}{P_{K}x} = \gamma$.
	However, because $P_{K}0 = 0$,
	we infer that $\gamma = 0$.
	Therefore,
	for every $x \in C$, 
	it follows from 
	\cref{l:cone-identities}\cref{i:cone-identities1}
	that 
	\begin{math}
	\norm{P_{K}x}^{2} =  \scal{x}{P_{K}x}
	=\scal{P_{C}x}{P_{K}x} = \gamma = 0; 
	\end{math}
	hence, $P_{K}x = 0$, and
	\cite[Theorem~6.30(i)]{bauschke2017convex}
	thus implies that $x = P_{\pc{K}}x \in \pc{K}.$
	Consequently, $C \subseteq \pc{K}$,
	as claimed.
\end{proof}

The following example shows that
the conclusion of \cref{c:set-cone}
is merely a necessary condition 
for $P_{C}  + P_{K} = P_{C+K}$
even when $C$ is a cone.

\begin{example}\label{eg:counter-cone-set2}
	Suppose that $\HH = \RR^{2}$.
	Set $u \coloneqq \rb{1,0}$
	and $v \coloneqq \rb{-1,1}$.
	Moreover, 
	set $C\coloneqq \RP v $
	and $K \coloneqq \RP u.$
	Then,
	because $\scal{u}{v} = -1 < 0$, 
	we see that 
	$C\subseteq \pc{K}$.
	Furthermore, $C+K$
	is a closed cone
	by \cite[Proposition~6.8]{bauschke2017convex}.
	Now set $x \coloneqq \rb{1,1} = v + 2u \in C + K$. 
	According to 
	\cite[Example~29.31]{bauschke2017convex},
	$P_{C}x + P_{K}x = \rb{0,0} + \rb{1,0} = \rb{1,0}
	\neq x = P_{C+K}x.$
	Therefore, 
	 $P_{C} + P_{K} \neq P_{C+K}.$
\end{example}

We now extend the classical 
\cite[Theorem~2, p.~46]{Halmos-spectral-1951}
(in the case of two subspaces)
by replacing
one subspace by a general convex set.

\begin{corollary}\label{c:subspace-sum}
	Let $C$ 
	be a nonempty closed convex subset of $\HH$,
	and let $V$ be a closed
	linear subspace of $\HH$.
	Then the following are equivalent:
	\begin{enumerate}
		\item\label{i:subspace-sum1} There exists
		a closed convex set 
		$D$ such that
		$P_{C} + P_{V} = P_{D}$.
		\item\label{i:subspace-sum2} $C \perp V.$
	\end{enumerate}
	Moreover, if \cref{i:subspace-sum1} and \cref{i:subspace-sum2} hold,
	then $D=C+V$
	and $P_{C} + P_{V} = P_{C+V}$.
\end{corollary}

\begin{proof}
	``\cref{i:subspace-sum1}\ensuremath{\implies}\cref{i:subspace-sum2}'':
	It follows from \cref{t:sum-pr} that $D=C+V$ 
	and that  $P_{C} + P_{V} = P_{C+V}$. 
	Now, by \cref{c:set-cone} and 
	\cite[Proposition~6.23]{bauschke2017convex}, 
	we obtain
	$C \subseteq \pc{V} = V^{\perp}$.
	
	``\cref{i:subspace-sum2}\ensuremath{\implies}\cref{i:subspace-sum1}'':
	Immediate from \cref{c:CperpD}.
\end{proof}

However, 
replacing 
the subspace $V$ 
in \cref{c:subspace-sum}
by cone might not work. 
The following
simple example shows
that 
the implication 
``\cref{i:subspace-sum1}\ensuremath{\implies}\cref{i:subspace-sum2}''
of \cref{c:subspace-sum} may fail
even when $C$ and $V$ are cones.

\begin{example}\label{eg:two-rays}
	Consider the setting of 
	\cref{eg:two-rays2}.
	We have seen that 
	$P_{U} + P_{V} = P_{\RR w}$. 
	Yet, $U$ is not perpendicular to $V$.
	In fact, $\lspan U = \lspan V = \RR w$.
\end{example}

Combining \cref{c:cone-many},
\cref{t:two-cones-pr},
and \cref{c:subspace-sum},
we obtain the
following well-known result;
see \cite[Theorem~2, p.~46]{Halmos-spectral-1951}.

\begin{corollary}\label{c:many-subspaces}
	Let $\fa{V_{i}}{i \in I}$
	be a finite 
	family of closed
	linear subspaces of $\HH$. 
	Then $\sum_{ i \in I} P_{V_{i}}$
	is a projector associated
	with  
	a closed linear subspace
	if and only if,
	for every $\rb{i,j} \in I \times I$
	with $i \neq j$,
	we have $V_{i} \perp V_{j}$.
\end{corollary}

\paragraph{Acknowledgments}{   
	The authors thank two referees for their constructive
	comments. We also
	thank Professors Rebecca Tyson
	and Chris Cosner
	for helpful comments on \cref{rem:R2}.
      We are grateful to
      Professor
      Patrick Combettes
      for bringing our attention
    to the case of convex averages.
	HHB and XW were partially 
	supported by NSERC Discovery Grants;
	MNB was partially supported 
	by a Mitacs Globalink Graduate Fellowship Award.
      Most parts of this work
      were done when MNB was
      a graduate student at
      the University of British Columbia, 
    Okanagan campus.
}




\begin{thebibliography}{20}

\bibitem{Bardi-Dolcetta-1997}
{\sc M.~Bardi and I.~C. Dolcetta}, {\em {O}ptimal {C}ontrol and {V}iscosity
  {S}olutions of {H}amilton-{J}acobi-{B}ellman {E}quations}, Birkh\"{a}user,
  Basel, 1997.

\bibitem{bartz2017resolvent}
{\sc S.~Bartz, H.~H. Bauschke, and X.~Wang}, {\em The resolvent order: a
  unification of the orders by {Z}arantonello, by {L}oewner, and by {M}oreau},
  SIAM J. Optim., 27 (2017), pp.~466--477.

\bibitem{bauschke2017projecting}
{\sc H.~H. Bauschke, M.~N. Bui, and X.~Wang}, {\em Projecting onto the
intersection of a cone and a sphere}, SIAM J. Optim.,
\newblock to appear (\url{https://arxiv.org/abs/1708.00585}).

\bibitem{bauschke2017convex}
{\sc H.~H. Bauschke and P.~L. Combettes}, {\em {C}onvex {A}nalysis and
  {M}onotone {O}perator {T}heory in {H}ilbert {S}paces}, Springer, New York,
  second~ed., 2017.

\bibitem{bauschke2004finding}
{\sc H.~H. Bauschke, P.~L. Combettes, and D.~R. Luke}, {\em Finding best
  approximation pairs relative to two closed convex sets in {H}ilbert spaces},
  J. Approx. Theory, 127 (2004), pp.~178--192.

\bibitem{bauschke2006strongly}
{\sc H.~H. Bauschke, P.~L. Combettes, and D.~R. Luke}, {\em A strongly
  convergent reflection method for finding the projection onto the intersection
  of two closed convex sets in a {H}ilbert space}, J. Approx. Theory, 141
  (2006), pp.~63--69.

\bibitem{Bauschke-Wang-Moffat-2013}
{\sc H.~H. Bauschke, S.~M. Moffat, and X.~Wang}, {\em Near equality, near
  convexity, sums of maximally monotone operators, and averages of firmly
  nonexpansive mappings}, Math. Program., Ser. B, 139 (2013), pp.~55--70.

\bibitem{bauschke-moursi-2016}
{\sc H.~H. Bauschke and W.~M. Moursi}, {\em The {Douglas--Rachford} algorithm
  for two (not necessarily intersecting) affine subspaces}, SIAM J. Optim, 26
  (2016), pp.~968--985.

\bibitem{bauschke-walaa-dec2017}
{\sc H.~H. Bauschke and W.~M. Moursi}, {\em The magnitude of the minimal
  displacement vector for compositions and convex combinations of firmly
  nonexpansive mappings}, Optim. Lett.,  (2018).
\newblock \url{https://doi.org/10.1007/s11590-018-1259-5}.

\bibitem{coleman2012calculus}
{\sc R.~Coleman}, {\em {C}alculus on {N}ormed {V}ector {S}paces}, Springer, New
  York, 2012.

\bibitem{Combettes-2018-mono}
{\sc P.~L. Combettes}, {\em Monotone operator theory in convex optimization},
  Math. Program., Ser. B, 170 (2018), pp.~177--206.

\bibitem{Cosner}
{\sc C.~Cosner}, {\em Personal communication}, 2018.

\bibitem{DenkowskiMigorskiPapageorgiou-nonlinear}
{\sc Z.~Denkowski, S.~Mig\'orski, and N.~S. Papageorgiou}, {\em {A}n
  {I}ntroduction to {N}onlinear {A}nalysis: {T}heory}, Kluwer, Dordrecht, 2003.

\bibitem{Halmos-spectral-1951}
{\sc P.~R. Halmos}, {\em Introduction to {H}ilbert {S}pace and the {T}heory of
  {S}pectral {M}ultiplicity}, Chelsea Publishing Company, New York, 1951.

\bibitem{hiriart2013convex}
{\sc J.-B. Hiriart-Urruty and C.~Lemar\'echal}, {\em Convex {A}nalysis and
  {M}inimization {A}lgorithms {I}}, Springer-Verlag, Berlin, 1993.

\bibitem{geogebra}
{\sc {International GeoGebra Institute}}, {\em Geogebra software}.
\newblock \url{https://www.geogebra.org/}.

\bibitem{KayeQuantum}
{\sc P.~Kaye, R.~Laflamme, and M.~Mosca}, {\em An {I}ntroduction to {Q}uantum
  {C}omputing}, Oxford University Press, New York, 2007.

\bibitem{Boris-I}
{\sc B.~S. Mordukhovich}, {\em {V}ariational {A}nalysis and {G}eneralized
  {D}ifferentiation I -- {B}asic {T}heory}, Springer-Verlag Berlin, 2006.

\bibitem{Moreau-cone-1962}
{\sc J.-J. Moreau}, {\em D\'{e}composition orthogonale d'un espace hilbertien
  selon deux c\^{o}nes mutuellement polaires}, C. R. Acad. Sci. Paris S\'{e}r.
  A, 255 (1962), pp.~238--240.

\bibitem{Moreau-prox-1963}
\leavevmode\vrule height 2pt depth -1.6pt width 23pt, {\em Propri\'{e}t\'{e}s
  des applications ``prox''}, C. R. Acad. Sci. Paris S\'{e}r. A, 256 (1963),
  pp.~1069--1071.

\bibitem{Moreau-prox-dual-1965}
\leavevmode\vrule height 2pt depth -1.6pt width 23pt, {\em Proximit\'{e} et
  dualit\'{e} dans un espace hilbertien}, Bull. Soc. Math. France, 93 (1965),
  pp.~273--299.

\bibitem{Moreau-convex-1966}
\leavevmode\vrule height 2pt depth -1.6pt width 23pt, {\em Fonctionnelles
  convexes}, S\'{e}minaire Jean Leray sur les \'{E}quations aux
  D\'{e}riv\'{e}es Partielles, no. 2, Coll\`{e}ge de France, Paris,
  (1966--1967).
\newblock \url{http://www.numdam.org/item/SJL_1966-1967___2_1_0}, Corrected
  reprint: Edizioni del Dipartimento di Ingegneria Civile, Universit` a di Roma
  Tor Vergata, Rome, 2003.

\bibitem{rocky}
{\sc R.~T. Rockafellar}, {\em Convex {A}nalysis}, Princeton University Press,
  1970.

\bibitem{weidmann2012}
{\sc J.~Weidmann}, {\em {L}inear {O}perators in {H}ilbert {S}paces}, Springer,
  New York, 1980.

\bibitem{Zalinescu-book2002}
{\sc C.~Z{\u{a}}linescu}, {\em Convex {A}nalysis in {G}eneral {V}ector
  {S}paces}, World Scientific Publishing Co. Inc., River Edge, NJ, 2002.

\bibitem{zarontello1971projections-partI}
{\sc E.~H. Zarantonello}, {\em Projections on convex sets in {H}ilbert space
  and spectral theory. {I}. {P}rojections on convex sets}, in Contributions to
  Nonlinear Functional Analysis, E.~H. Zarantonello, ed., Academic Press, New
  York, 1971, pp.~237--341.

\bibitem{zarontello1971projectionsII}
\leavevmode\vrule height 2pt depth -1.6pt width 23pt, {\em Projections on
  convex sets in {H}ilbert space and spectral theory. {II}. {S}pectral theory},
  in Contributions to Nonlinear Functional Analysis, E.~H. Zarantonello, ed.,
  Academic Press, New York, 1971, pp.~343--424.

\end{thebibliography}

\end{document}